\documentclass[11pt]{amsart}

\usepackage{amsmath}
\usepackage{bbm}
\usepackage{mathrsfs}
\usepackage[english]{babel}
\usepackage[utf8x]{inputenc}
\usepackage[T1]{fontenc}

\usepackage[top=3cm,bottom=2cm,left=3cm,right=3cm,marginparwidth=1.75cm]{geometry}

\usepackage{amsmath}
\usepackage{graphicx}
\usepackage[colorinlistoftodos]{todonotes}
\usepackage[colorlinks=true, allcolors=blue]{hyperref}
\usepackage{xypic}
\usepackage{tabu}
\usepackage{comment}
\usepackage{stmaryrd} 
\usepackage{xcolor}
\usepackage{tikz}
\usepackage{tikzcd}

\usepackage{amssymb}
\usepackage{tikz-cd}
\usetikzlibrary{cd}
\usepackage{relsize} 

\newcommand{\bigint}{{\mathlarger{\mathlarger{\int}}}}
\newcommand{\Qp}{\mathbb{Q}_p}
\newcommand{\Zp}{\mathbb{Z}_p}
\newcommand{\Qpalg}{\overline{\mathbb{Q}}_p}

\newcommand{\abs}[1]{\left\lvert#1\right\rvert_p }
\newcommand{\norm}[1]{\left\lVert#1\right\rVert_p }
\newcommand{\ssq}{\subseteq}
\newcommand{\be}{\begin{equation}}
\newcommand{\ee}{\end{equation}}

\newcommand{\bs}{\backslash}
\newcommand{\tail}{\mathrm{tail}}

\DeclareMathOperator{\codim}{codim}
\DeclareMathOperator{\Spec}{Spec}
\DeclareMathOperator{\GL}{GL}
\DeclareMathOperator{\Fix}{Fix}
\DeclareMathOperator{\vol}{vol}

\usepackage{mathtools} 
\mathtoolsset{showonlyrefs,showmanualtags}
\numberwithin{equation}{section}
\usepackage{amsmath}

\renewcommand{\:}{\colon}
\renewcommand{\P}{\mathbb{P}}

\newtheorem{theorem}{Theorem}
\newtheorem{lemma}[theorem]{Lemma}
\newtheorem{proposition}[theorem]{Proposition}
\newtheorem{corollary}[theorem]{Corollary}

\theoremstyle{definition}
\newtheorem{definition}[theorem]{Definition}

\newtheorem{example}[theorem]{Example}
\newtheorem*{warning*}{Warning!}

\theoremstyle{remark} 
\newtheorem{remark}[theorem]{Remark}

\title{p--adic integral geometry}
\author{Avinash Kulkarni}
\address{Max Planck Institute MIS Leipzig}
\email{avinash@mis.mpg.de}
\author{Antonio Lerario}
\date{}
\address{SISSA (Trieste)}
\email{lerario@sissa.it}


\begin{document}

\makeatletter
\tagsleft@false
\makeatother

\maketitle
\begin{abstract}We prove a p-adic version of the Integral Geometry Formula for averaging the intersection of two p-adic projective algebraic sets. We apply this result to give bounds on the number of points in the modulo $p^m$ reduction of a projective set (reproving a result by Oesterl\'e) and to the study of random p-adic polynomial systems of equations. 

\end{abstract}


\section{Introduction}
Integral Geometry deals with ``averaging'' metric properties (e.g. the volume) of the intersection of two submanifolds of a homogeneous space under the action of a Lie group. This classical subject has a vast range of applications and connections to different areas of mathematics, including differential geometry \cite{Howard, Santalo, Alesker}, representation theory \cite{Bernig, Howard}, convex geometry \cite{Schneider1, Schenider2, Weil}, numerical anlysis \cite{AmelunxenBurgisser, Condition}, and random geometry \cite{EdelmanKostlan, Symmetric}. 
The main goal of this paper is the adaptation of the ideas coming from this subject to the $p$-adic world. Before moving to the details, let us spend a few words discussing the classical framework.

\subsection{Classical Integral Geometry}\label{sec:classIGF}The main reference we have in mind for this section is the monograph \cite{Howard}, where the spherical case is discussed; here we prefer to discuss the projective case as it offers a closer analogy with the results of this article.

Let $A, B\subset {\mathbb{R}\P}^n$ be two compact submanifolds of dimension $\dim(A)=a$ and $\dim(B)=b$. We endow the projective space with the quotient metric from the unit sphere $S^n\subset \mathbb{R}^{n+1}$ and, restricting this metric to a submanifold $Y$ of dimension $\dim(Y)=k$, we get a natural notion of Riemannian $k$-dimensional volume of $Y$, which we denote by $\mathrm{vol}_k(Y)$. 
More generally, if $Y$ is a codimension $m-k$ submanifold of a Riemannian manifold $M$ whose volume density is $\mathrm{vol}$ and $U(Y, \epsilon)=\bigcup_{x\in Y}B(x, \epsilon)$ denotes the $\epsilon$-neighborhood of $Y$ in $M$, then
    \be \label{eq:volume}
        \mathrm{vol}_k(Y):=\lim_{\epsilon\to 0} \frac{\mathrm{vol}(U(Y, \epsilon))}{\mathrm{vol}(B_{\mathbb{R}^{m-k}}(0, \epsilon))}.
    \ee
The volume of the projective space ${\mathbb{R}\P}^k\subseteq {\mathbb{R}\P}^n$ with the quotient metric is 
    \be 
        \mathrm{vol}_k({\mathbb{R}\P}^k)=\frac{\pi^{\frac{k+1}{2}}}{\Gamma\left(\frac{k+1}{2}\right)}.
    \ee
The group $G=O(n+1)$ acts on the projective space by isometries, and by a standard transversality argument, the intersection $A\cap g B$ is transversal and of codimension ${(n-a)+(n-b)}$, i.e. of dimension $k=n-(n-a)-(n-b)$, for almost every $g\in G$. The Integral Geometry Formula \cite[Corollary 3.9]{Howard} then tells that:
\be\label{eq:IGFstd}
    \int_G \frac{\vol_k(A\cap gB)}{\vol_k(\mathbb{R}\P^k)}\, dg=\frac{\vol_a(A)}{\vol_a(\mathbb{R}\P^a)}\cdot \frac{\vol_b(B)}{\vol_b(\mathbb{R}\P^b)},
\ee
where the integral is with respect to the normalized Haar measure $\int_G dg=1$, and the integrand is finite for almost every $g\in G$ (by transversality).
The identity \eqref{eq:IGFstd} is very remarkable and has many interesting applications. The spirit of this paper is much influenced by one of these applications in particular, due to \cite{EdelmanKostlan}, where the idea of combining \eqref{eq:IGFstd} with algebraic geometry and Example~\ref{example:randomreal} first appeared.

\begin{example}[How many zeros of a random polynomial are real?]\label{example:randomreal}
    Let $f\in \mathbb{R}[x_0, x_1]_{(d)}$ be the random polynomial given by:
        \be\label{eq:Kostlanreal} 
            f(x_0, x_1)=\xi_0x_0^d+\cdots +\xi_k\sqrt{{\binom{d}{k}}}x_0^{d-k}x_1^k+\cdots +\xi_d x_1^d,
        \ee
    where $\{\xi_k\}_{k=0}^d$ is a family of i.i.d. standard Gaussians. In \cite{EdelmanKostlan} the authors use a nice argument from integral geometry to show that the expectation of the number of zeroes on ${\mathbb{R}\P}^1$ of the random polynomial \eqref{eq:Kostlanreal} is $\sqrt{d}$. The argument goes as follows. One first considers the \emph{Veronese} embedding $\nu:{\mathbb{R}\P}^1\to {\mathbb{R}\P}^d$ given by:
        \be \nu:[x_0, x_1]\mapsto \left[x_0^d, \ldots, \sqrt{{\binom{d}{k}}}x_0^{d-k}x_1^k,\ldots, x_1^d\right].
        \ee
    Using the properties of the Gaussian distribution, one then proves that the average number of zeroes of \eqref{eq:Kostlanreal} equals the average number of points of intersection of $\nu({\mathbb{R}\P}^1)$ with a randomly sampled hyperplane $g{\mathbb{R}\P}^{n-1}\subset {\mathbb{R}\P}^n$, with $g\in O(n+1)$ drawn from a uniform distribution. Finally, by \eqref{eq:IGFstd}, this expectation equals
        \be\label{eq:expveroreal}         \mathbb{E}\#\{f=0\}=\frac{\mathrm{vol}_1(\nu({\mathbb{R}\P}^1))}{\mathrm{vol}_1({\mathbb{R}\P}^1)}=\sqrt{d}.
        \ee
    In other words, by using the integral geometry formula, we see that the expectation of the number of zeros of a random polynomial with respect to a particular distribution is determined by the length of the corresponding Veronese curve. We observe, however, that computing the length of a Veronese curve is in general not trivial. In Theorem~\ref{thm:randompoly} we will prove the $p$-adic version of the results discussed in this example.
\end{example}

\subsection{\texorpdfstring{$p$}{p}-adic Integral Geometry}

Keeping Section \ref{sec:classIGF} in mind, let us discuss now the content of the current paper. We begin with the geometry of the ambient space. We denote by $\Qp$ the set field $p$-adic numbers, endowed with the standard $p$-adic norm, and by $\Zp$ the ring of $p$-adic integers. The vector space $\Qp^{n+1}$ is endowed with the norm:
\be \norm{(a_0, \ldots, a_n)} =\sup_i{|a_i|_p},\quad (a_0, \ldots, a_n)\in \Qp^{n+1}.\ee
Together with the metric structure, we also consider the standard $p$-adic measure $\mu$, defined on the Borel sets of $\Qp^{n+1}$ and normalized by $\mu(\Zp^{n+1})=1$, see Section \ref{sec:measure}.

We also turn the $p$-adic projective space $\P^n$ into a metric measure space as follows. First we consider the Hopf fibration $\varphi:S^n\to \P^n$, where $S^n$ is the open and compact subset of $\Qp^{n+1}$ consisting of points $a=(a_0, \ldots, a_n)\in \Qp^{n+1}$ such that $\norm{a}=1$ and $\varphi$ is the natural projection map. Unlike the real case, the $p$-adic unit sphere as a $p$-adic analytic manifold \emph{is not} of dimension $n$, but of dimension $n+1$ (see Figure \ref{fig:3adic unit sphere}). Nevertheless, we have decided to retain the same notation due to the strong parallels with the real case.

    The metric structure on $\P^n$ is then defined by:
    \be \label{eq:metricproj} d(x,y)= \norm{\hat x\wedge \hat y}, \quad x,y\in \P^n,\ee
where $\hat x\in \varphi^{-1}(x)$ and $\hat y\in \varphi^{-1}(y)$ are any two preimages and $\hat x\wedge \hat y\in \Lambda^{2}(\Qp^{n+1})$. For the measure, we take the normalized pushforward $\frac{\varphi_*\mu}{1+p^{-1}}$. When there is no risk of ambiguity we will still denote this measure by $\mu$.

The main ingredients to formulate an identity in the style of \eqref{eq:IGFstd} are (1) the action on a topological space $X$ of a topological compact group $G$ endowed with a measure and (2) a notion of ``volume'' and ``dimension'' for subspaces of $X$.
In this paper the group will be $G=\GL_{n+1}(\Zp)$, which acts by isometries and measure preserving transformations on $X=\P^n$. For the sake of this introductory section we will take $X=\P^n$, but the results we will present are also valid for the $p$-adic sphere $S^n\subset \Qp^{n+1}$ (see Section \ref{sec:strat}). We endow $\GL_{n+1}(\Zp)$ with the normalized Haar measure:
\be \int_{\GL_{n+1}(\Zp)}dg=1.\ee
 The class of submanifolds of $\P^n$ that we will consider consists of algebraic sets (but this condition can also be weakened to consider stratified $\Qp$-analytic sets, see Section \ref{sec:strat}). In this case the \emph{dimension} of an algebraic set $A\subseteq\P^n$ is the one coming from algebraic geometry; the notion of \emph{volume} is instead more interesting. Denoting by $B(x, r)$ the ball of radius $r$ centered at $x\in \P^n$ (with respect to the metric \eqref{eq:metricproj}), for an algebraic set $Y\subseteq \P^n$ of dimension $\dim(Y)=k$ we define
 its volume, based on the work of \cite{Serre} and \cite{Oesterle}, by:
    \be \label{eq:pvol} 
        \textrm{vol}_k(Y):=\lim_{m\to \infty}p^{m(n-k)}\cdot \mu\left(\bigcup_{x\in Y}B(x, p^{-m})\right).
    \ee
Notice that this definition mimics the classical definition given in \eqref{eq:volume}.
 The volume of the $p$-adic projective space equals (via Section \ref{sec:measure}):
 \be \mathrm{vol}_k(\P^k)=\frac{1-p^{-(k+1)}}{1-p^{-1}}.\ee
    \begin{figure}
        \centering
        \newcommand{\cellsize}{2.6} 

        \begin{tikzpicture}
        [
            box/.style={rectangle,draw=black,thick, minimum size= \cellsize cm},
        ]
        
        \foreach \x in {-1,0,1}{
            \foreach \y in {-1,0,1}
                \node[box] at (\cellsize*\x, \cellsize*\y)
                    { \begin{tiny} $ (\x,\y) + (3\mathbb{Z}_3)^2$ \end{tiny}};
        }
        \node[box,fill=lightgray] at (0,0){ \begin{tiny} $(0,0) + (\mathbb{Z}_3)^2$ \end{tiny} };  
        \end{tikzpicture}
    
        \caption{
            A depiction of the $3$-adic unit circle $S^1$ in $\mathbb{Z}_3^2$. Each cell represents an open ball of radius $\frac{1}{3}$ centered at the indicated point. The union of all the unshaded cells is $S^1$. The unit $-1 \in \mathbb{Z}_3^\times$ acts by reflecting the diagram through the origin. Note that $S^1$ has non-zero measure inside $\mathbb{Z}_3^2$, as well as the pullback $\varphi^{-1}(U)$ of any non-empty open subset of $\P^1$.
        } 
        \label{fig:3adic unit sphere}
    \end{figure}
 The main theorem around which this paper revolves is Theorem~\ref{thm:IGF}, restated below. 
 
  \begin{theorem}[$p$-adic Integral Geometry Formula]\label{thm:pIGF}
        Let $ {A},  {B}\subset \P^{n} $ be algebraic sets of dimensions $\dim( {A})=a, \dim( {B})=b.$ Then for almost all $g\in \GL _{n+1}(\Zp)$ the intersection $ {A}\cap g {B}$ is transversal, of dimension $k=n-(n-a)-(n-b)$, and:
        \be \label{eq:pIGF}\int_{\mathrm{GL}_{n+1}(\Zp)}\frac{\mathrm{vol}_k\left( {A}\cap g {B}\right)}{\mathrm{vol}_k\left(\P^k\right)} dg=\frac{\mathrm{vol}_a( {A})}{\mathrm{vol}_a(\P^a)}\cdot \frac{\mathrm{vol}_b( {B})}{\mathrm{vol}_b(\P^b)}.\ee
    \end{theorem}
    
 Notice the exact analogy between \eqref{eq:IGFstd} and \eqref{eq:pIGF}. Before presenting a couple of applications, let us commentate on the notion of volume \eqref{eq:pvol} that we use in Theorem~\ref{thm:pIGF}.
 
 First, since $Y$ sits inside the metric space $\P^n$, it is possible to define its Hausdorff dimension, which coincides with its dimension $k=\dim(Y)$ as an algebraic set, and its $k$-dimensional Hausdorff measure turns out to be equal to its $k$-dimensional volume.
 
 Second, a result of Oesterl\'e \cite[Theorem 2]{Oesterle} connects the number of congruence classes of points mod $p^m$ in $Y$ and the volume of $Y$. Specifically, denoting by $N_m(Y)$ the number of points of $Y$ in its modulo-$p^m$ reduction (see Definition \ref{def:modulo}), we prove that:
 \be\label{eq:limvolintro} \mathrm{vol}_k(Y)=\lim_{m\to \infty}\frac{N_m(Y)}{p^{mk}}.\ee
 In fact, if $Y$ is smooth, then the limit above ``stabilizes'' for large enough $m$. The aforementioned result, as well as the finiteness of the limit above, were both proven by Serre in \cite{Serre}. We give a simplified proof of Oesterl\'e's result, which we state as Corollary~\ref{cor:volumepoint} and Corollary~\ref{cor: point counts for integral smooth compact manifolds}.
 
 Third, all this discussion on volumes and the $p$-adic Integral Geometry Formula is valid for open compact subsets of $Y$ as well as for the subset of smooth points of $Y$. The volume of the smooth sublocus of $Y$ and the volume of $Y$ are equal, see Corollary \ref{cor:volumepoint} and Corollary \ref{coro:smooth}.
 
 Fourth, by taking $B\subset \P^n$ to be a subspace of complementary dimension $B=\P^{n-k}$, we see by the $p$-adic Integral Geometry formula that the volume of $Y$ satisfies the identity:
    \be \label{eq:degreeint} 
        \mathrm{vol}_k(Y)=\mathrm{vol}_{k}(\P^k)\cdot \int_{\GL_{n+1}(\Zp)}\# \left(Y\cap g\P^{n-k}\right)dg.
    \ee
 This last observation allows us to prove a slightly weaker version of a result of Oesterl\'e \cite[Theorem 1]{Oesterle}, though with a substantially simpler proof. Oesterl\'e's result provides an estimate for the number of points in the modulo-$p^m$ reduction of a $\Zp$-algebraic set. Here again we state the projective version of this result.
 
 \begin{corollary} \label{cor: Oesterle Theorem 1}
    Let $Y\subseteq \P^n$ be an algebraic set of dimension $k$ and degree $d$. Then $\vol_k(Y) \leq d$. Moreover, if $\vol_k(Y) < d$, then for $m$ large enough:
    \be 
        N_m(Y)\leq d\cdot  p^{mk} \mathrm{vol}_k(\P^k).
    \ee
 \end{corollary}
 
 \begin{proof}
    Note that the integrand of \ref{eq:degreeint} is bounded by the degree for almost all $g \in \GL_{n+1}(\Zp)$. If $\vol_k(Y) < d$, then the result follows immediately.
 \end{proof}
 
 \begin{remark}
    The second conclusion of the result above holds under the hypothesis that $Y$ is equidimensional. If $Y$ is not the union of $\deg Y$ linear subspaces, then there is a smooth point $y \in Y(\Qp)$ such that the tangent space $H := T_yY$ intersects $Y$ with multiplicity exactly $2$ at $y$. By perturbing $H$, we can construct a linear subspace $H'$ such that $\#(H' \cap Y) < d-2$. More precisely, such that at least two points of $H'\cap Y$ are defined over a ramified extension of $\Qp$. In particular, the interior of the subset
        \[
            U := \{ g \in \GL_{n+1}(\Zp) : \# (Y \cap g\P^{n-k}) < d \}.
        \]
    is non-empty, and thus we have that $\vol_k(Y) < d$. The claim then follows from Corollary~\ref{cor: Oesterle Theorem 1}.
    
    Conversely, the hypothesis that $Y$ is equidimensional is necessary. For instance, for $Y$ the union of $d$ lines and $a$ points in $\P^2$, then $N_m(Y)$ over-estimates the volume for almost all $m$.
\end{remark}
 
\subsection{Applications}

 We conclude this introductory section with an application of the $p$-adic Integral Geometry Formula to random $p$-adic polynomials, in a spirit similar to Example \ref{example:randomreal} above. In the paper \cite{Evans}, Evans considers the problem of studying the expectation of the number of zeros in $\Zp^n$ of a system of random $p$-adic polynomials. The distribution that Evans considers comes from picking Gaussian $p$-adic coefficients
 in front of the Mahler basis (see \cite{Evans} for more details on the notion of $p$-adic gaussian variable). When $n=1$, this means considering the following random polynomial:
 \be\label{eq:evanspoly}
 g(t) := \zeta_0 + \zeta_1 \binom{t}{1} + \ldots + \zeta_d \binom{t}{d},
 \ee
      where $\{\zeta_k\}_{k=0}^d$ is a family of i.i.d. uniform variables in $\Zp.$
   Evans computed the expected number of zeroes of $g$ in $\Zp$; using our techniques we recover and extend this result, computing the expectation of the number of zeroes on the whole $p$-adic line. More precisely, in Section~\ref{sec:mahler} we prove the following result.
 \begin{theorem} \label{Evans univariate}
    Let $g(t)$ be the random $p$-adic polynomial \eqref{eq:evanspoly}.
    \begin{enumerate}
 \item The expected number of zeroes in $\Zp$ of $g$ equals $\frac{p^{\lfloor \log_p d\rfloor}}{1+p^{-1}}$ (this is Evans result).
 \item For every $m\geq 1$ the expected number of zeroes of $g$ in the annulus $\frac{1}{p^m}\Zp\backslash \frac{1}{p^{m-1}}\Zp$ equals $\frac{|d|_p}{p^m}\frac{1-p^{-1}}{1+p^{-1}}$ and in particular:
 \be \mathbb{E}\#\{g=0\}=\frac{p^{\lfloor \log_p d\rfloor}}{1+p^{-1}}+\sum_{m\geq 1}\frac{|d|_p}{p^m}\frac{1-p^{-1}}{1+p^{-1}}=\frac{p^{\lfloor \log_p d\rfloor}+|d|_p {p^{-1}}}{1+p^{-1}}.
 \ee
 \end{enumerate}
 \end{theorem}

In this paper we also consider a different model that seems more natural over the projecive space, due to its invariance under the group $\GL_{n+1}(\Zp)$. In the case $n=1$, we define the random $p$-adic polynomial:
    \be \label{eq:randompadic}
        f(x_0, x_1)=\zeta_0x_0^d+\cdots+\zeta_k x_0^{d-k} x_1^{k}+\cdots+\zeta_d x_1^d,
    \ee
 where $\{\zeta_k\}_{k=0}^d$ is a family of i.i.d. uniform variables in $\Zp$. This distribution is invariant under $\GL_2(\Zp)$-change of variables -- there are no preferred points or directions in the projective line $\P^1$. For the random polynomial \eqref{eq:randompadic} we prove the following (surprising) result.
 
 \begin{proposition}
    The expected number of zeroes of the polynomial \eqref{eq:randompadic} in $\P^1$ is $1$. Moreover the density of zeroes on $\P^1$ is uniform.
 \end{proposition}
 
 In fact, we prove a much more general theorem which relates the expectation of the number of zeroes of a system of random $p$-adic equations with the volume of the image of $\P^n$ under an appropriate ``Veronese'' map (as in \cite{EdelmanKostlan}). In the case of the polynomial \eqref{eq:randompadic} the Veronese map to consider is $\nu:\P^1\to \P^d$ given by $[x_0, x_1]\mapsto [x_0^d, \cdots, x_0^{d-k}x_1^k, \ldots, x_1^d]$
 and, in analogy with \eqref{eq:expveroreal}, we have:
 \be \mathbb{E}\#\{f=0\}=\frac{\mathrm{vol}_1(\nu(\P^1))}{\mathrm{vol}_1(\P^1)}=1.\ee
 Once again, the computation of the length of the Veronese curve is an essential step to obtain the result. For us, this is given in Proposition~\ref{propo:veronese}.
 \subsection*{Acknowledgements}The authors want to thank the Max Planck Institute for Mathematics in the Sciences (MiS) in Leipzig, for the stimulating atmosphere where this work was done. Special thanks also to Paul Breiding, Peter B\"urgisser, and Bernd Sturmfels for helpful comments and valuable discussions.\section{Preliminaries}

\subsection{Notation} \ \\
    
    \begin{tabular}{rcl}
    $\abs{\cdot}$   &--& The $p$-adic absolute value, with normalization $\abs{p} = p^{-1}$.\\
    $\norm{\cdot}$  &--& The norm of a vector in $\Qp^{n+1}$ \\
        $R_m$  &--& The ring $\mathbb{Z}/p^m\mathbb{Z}$ of integers modulo $p^m$. \\
$\P^n$          &--& Projective space over $\Qp$ \\
    $B(a; p^{-m})$  &--& The ball in a metric space centered at $a$ of radius $p^{-m}$\\
    $N_m(U)$        &--& The minumum number of balls of radius $p^{-m}$ needed to cover $U$\\ &&(alternatively: the number of points in the modulo $p^m$ reduction of $U$)\\
    $S^n$           &--& The unit sphere in $\Qp^{n+1}$ \\
    $\abs{J_{\mathbf{f}}(a)}$ &--& The absolute value of the determinant of the Jacobian matrix for \\ & & the polynomials $\mathbf{f}$. (In the projective case, normalized by the norm of $a$.) \\
    $\varphi$ &--& The Hopf fibration $\varphi\colon S^n \rightarrow \P^n$
    \end{tabular}

\subsection{Affine and projective sets and their stratification}\label{sec:strat}
    
    A $\Qp$-analytic manifold is defined analogously to a real or complex analytic manifold. i.e, in terms of charts and gluing maps. For the reader interested in the technical details, we refer to \cite[Part II]{Schneider}. Similarly, if $\overline\Qp$ is an algebraic closure of $\Qp$, there is the corresponding notion of $\overline \Qp$-analytic manifold.
    
    For a smooth projective algebraic variety $ X$ defined over $\Qp$, the set of $\Zp$-points of $X$ has the structure of a $\Qp$-analytic manifold. It is necessary for us to consider algebro-geometric objects whose set of $\Zp$-points does not have the structure of a $\Qp$-analytic manifold. 
    
    \begin{definition}
        A \emph{stratified $\Qp$-analytic set} is a set $X$ which admits a decomposition into finitely many subsets (called \emph{strata})
            \[
                X := \coprod_{i=1}^k X_i
            \]
        such that each $X_i$ is a $\Qp$-analytic manifold. The \emph{dimension} of $X$ is the maximal dimension of the strata. If $X$ is empty we consider the empty coproduct and we set $\dim(X)=-\infty.$
    \end{definition}
    The previous notion of stratified set is weaker than the usual notion coming from semialgebraic geometry, singularity theory or algebraic geometry, but all we need in this paper is to be able to decompose the sets we are working with into finitely many disjoint pieces (the strata) each one of which is a $\Qp$-analytic manifold.
    
    Note that for a stratified $\Qp$-analytic set, the pieces need not be open or closed submanifolds of $X$, nor is the decomposition unique. Our central category of interest is the category of (analytically) open compact subsets of algebraic sets. At this point, it is conveinent to define some notation we will use throught the paper.
    
    \begin{definition}
        Let $X \ssq \Qp^n$ be an algebraic set of codimension $r$, and let $a \in X$ be a smooth point. Let $\mathbf{f} := (f_1, \ldots, f_r)$ be local equations for $X$ at $a$. Then we define
        \[
            \abs{J_\mathbf{f}(a)} := \sup_{I} \left\{ \abs{M_{\mathbf{f},I}(a)} : M_{\mathbf{f},I} \text{ is an $r \times r$ minor of } J_\mathbf{f}(a) \right\}.
        \] 
        Let $X \ssq \P^n$ be a projective algebraic set of codimension $r$, let $a \in X$ is a smooth point, and let $\mathbf{f} := (f_1, \ldots, f_r)$ be local equations for $X$ at $a$. With $\delta := \sum_{i=1}^r (\deg(f_i)-1)$, we define
        \[
                \abs{J_\mathbf{f}(a)} := \sup_{I} \left\{ \frac{\abs{M_{\mathbf{f},I}(a)}}{\norm{a}^\delta} : M_{\mathbf{f},I} \text{ is an $r \times r$ minor of } J_\mathbf{f}(a) \right\}.
        \]
    \end{definition}
    Note in the projective definition that $\abs{J_\mathbf{f}(a)}$ is independent of the choice of homogeneous coordinates of $a$. However, both definitions are sensitive to the choice of defining equations.
    
    \begin{remark}
        If $U \ssq X$ is an open compact subset of the smooth subset of $X$, then there is a finite set of points $a_1, \ldots, a_k$ and open subsets $U_1, \ldots, U_k$ of $U$ such that each $X \cap U_i$ is defined by the local equations at $a_i$ and such that $U = \bigcup_{i=1}^k U_i$. In particular, the function
            $
                a \mapsto \sup_{1\leq i \leq k} \abs{J_{\mathbf{f}_i}(a)}
            $
        is bounded away from $0$ on $U$.
    \end{remark}
    
    \begin{proposition} \label{proposition: algebraic implies stratified analytic}
        Let $X \ssq \Qp^n$ be an algebraic set. Then $X$ is a stratified $\Qp$-analytic set. 
    \end{proposition}
    
    \begin{proof}
        Write $X := Z(f_1, \ldots, f_r)$. Let $\Qpalg$ denote an algebraic closure of $\Qp$, and let $X(\Qpalg)$ be the common zero locus of $f_1, \ldots, f_r$ over $\Qpalg$. We may assume that $f_1, \ldots, f_r$ generate a radical ideal. The singular sublocus $ Y(\Qpalg))=\operatorname{Sing}(X(\Qpalg))$ of $X(\Qpalg)$ is defined by the vanishing of the maximal minors of the Jacobian matrix, and so has positive codimension in $X(\Qpalg)$. We may write $X(\Qpalg) := (X(\Qpalg) \backslash Y(\Qpalg)) \cup Y(\Qpalg)$, and $(X(\Qpalg) \backslash Y(\Qpalg))$ is a $\Qpalg$-analytic manifold. Iterating, we obtain a decomposition
            \be
                X(\Qpalg) := \coprod_{i=1}^k X_i(\Qpalg).
            \ee
        Note that a point in $X(\Qp)$ is smooth in the algebraic sense if and only if it is also a smooth point of $X(\Qpalg)$. As algebraic smoothness implies analytic smoothness, the result follows by considering the decomposition (with some strata possibly empty over $\Qp$):
        \begin{align*} 
        X=X(\Qp)=\coprod_{i=1}^kX_i(\Qp) \tag*{\qedhere} 
        \end{align*}
    \end{proof}
    
    \begin{remark}
        Note that $\P^n$ can be written as
            \[
                \P^n := \coprod_{r=0}^n \Qp^r
            \]
        and that this is a stratification. In particular, if $X$ is any projective algebraic set then each component of $X = \coprod_{r=0} (\Qp^r \cap X)$ is a stratified $\Qp$-analytic set. Alternatively, by restricting to affine local patches, we see the $\Qp$-points of any separated scheme of finite type over $\Qp$ is also a stratified $\Qp$-analytic set, and the assignment $X \mapsto X(\Qp)$ is functorial. However, $\P^n(R_m)$ is not equal to $\coprod_{r=0}^n (R_m)^r$ when $m>1$, since $R_m$ has a nontrivial ideal.

    \end{remark}

    For any projective variety $X$ over $\Zp$, we have that $X(\Zp) = X(\Qp)$. There is a natural map defined by
        \be 
        \begin{tabu}{cccc}
            \varphi\: & S^{n} &\rightarrow &\P^{n}(\Zp) \\
                      & (x_0, \ldots, x_n) & \mapsto & [x_0 , \ldots , x_n]
        \end{tabu}
        \ee
    which is well-defined since $(0,\ldots, 0) \not \in S^n$. Plainly, we have that $\varphi$ is a surjective, $\Qp$-analytic, and hence measurable map. We have 
            $\varphi^{-1}\{[x_0, \ldots ,x_n]\} = \{ (\lambda x_0, \ldots, \lambda x_n) : \lambda \in \Zp^\times \}.$
    Inspired by the real and complex settings, we give this map the name:
    \begin{definition}
        We call the map $\varphi$ the \emph{Hopf fibration}. 
    \end{definition}

\subsection{Modulo \texorpdfstring{$p^m$}{p^{}m} reduction}

     We now discuss the relations of our models and their reductions modulo $p^m$. Denote
            \[
                S^{n}(R_m) := \{ (x_0, \ldots, x_n) : x_i \in R_m \text{ and at least one } x_i \in R_m^\times \}.
            \]
    As the next result indicates, reduction modulo $p^m$ commutes with projectivizing coordinates.
    \begin{proposition} \label{prop: reduction commutes with projectivization}
        Let $m > 0$ be an integer, let $ U \ssq \P^n$ be any subset. The following diagram (of continuous functions) commutes:
       \be 
        \begin{tikzcd}
            \varphi^{-1}(U) \arrow[r, "\varphi"] \arrow[d, "\widehat{\pi}_m"'] & U \arrow[d, "\pi_m"] \\
            \widehat{\pi}_{m}(\varphi^{-1}(U)) \arrow[r, "\varphi_m"]          & \pi_{m}(U)          
        \end{tikzcd}
        \ee
        The map $\varphi_m$ is defined by restricting the natural map
            \[
            \begin{tabu}{rccc}
                \varphi_m\: & S^n(R_m) & \rightarrow & \P^n(R_m) \\ 
                & (x_0, \ldots, x_n) & \mapsto & [ x_0 , \ldots , x_n ].
            \end{tabu}
            \]
        Moreover, for any $x \in \P^n(R_m)$, we have 
        \be\label{eq:card}
            {\# \varphi_m^{-1}(x) = \#R_m^\times = p^m\left(1-\frac{1}{p}\right)}.
        \ee
    \end{proposition}
    
    \begin{proof}
        By definition, any $x \in U$ has at least one coordinate not in $p\Zp$. Since $\pi_m(\Zp^\times) = R_m^\times$ for all $m$, the result follows. 
    \end{proof}
        
    \begin{definition}[Number of points modulo $p^m$]\label{def:modulo} If $U\subseteq A\subseteq \Zp^n$ is a subset of an algebraic set, we define
        $N_m(U)=\#\widehat{\pi}_m(U)$.
        In the projective case, if $U\subseteq A\subseteq \P^n$ is a subset of an algebraic set, we define
        $N_m(U)=\#\pi_m(U)$.
    \end{definition}

    \begin{example}
        Proposition~\ref{prop: reduction commutes with projectivization} allows us to determine $N_m(\P^n)$ easily. We have that
            \[
            \# S^n(R_m) = (\#R_m)^{n+1} - (\# pR_m)^{n+1} = p^{m(n+1)} - p^{(m-1)(n+1)},
            \]
        as the maximal ideal of $R_m$ is generated by $p$. 
        In particular, using \eqref{eq:card}, we get:
            \[
            \#\P^n(R_m) = \frac{ p^{m(n+1)} - p^{(m-1)(n+1)} }{p^m - p^{m-1}} = \frac{ p^{mn} - p^{m(n-1)-1} }{1 - p^{-1}}. 
            \]
        Notice that this number is \emph{not} equal to $\#(R_m)^n + \ldots + \#R_m + 1$ in general. This explicitly demonstrates why $\P(R_m)$ is generally not equal to $\coprod_{r=0}^n (R_m)^r$.
    \end{example}
        
\subsection{Measures and metrics}\label{sec:measure}

    On projective space, we obtain a naturally defined quotient measure and quotient metric from the unit sphere (see for instance \cite{Choi}).
    
        \begin{definition}
            We define the (normalized) pushforward measure on projective space by
                \[
                    \mu( U ) := \frac{1}{\mu(\Zp^\times)} \cdot \mu( \varphi^{-1}(U)).
                \]
        \end{definition}
        
        \begin{definition}\label{def:metricproj}
            Define the metric $d(x,y)$ for $x,y \in \P^{n}$ by
                \be\label{eq:distproj}
                    d(x,y) := \norm{\hat x \wedge \hat y}
                \ee
            for any $\hat x \in \varphi^{-1}(x)$ and $\hat y \in \varphi^{-1}(y)$.
        \end{definition}
    The metric in Definition~\ref{def:metricproj} is the natural quotient metric, and any subset $A\subset \P^{n}$ can be turned into a metric space by restricting the metric to $A$. Our normalization of the pushforward measure ensures that for any $x \in \P^{n}$ and $m > 0$, the ball $B(x;p^{-m}) \ssq \P^{n}$ has measure $p^{-mn}$. When $m=0$, we have $B(x;1) = \P^n$ for any $x \in \P^n$, so 
            \[
                \mu(\P^n) = (1-p^{-1})^{-1}\mu(S^{n}) = \frac{1-p^{-(n+1)}}{1-p^{-1}}.
            \]    

    \subsection{Hensel's lemma}
        
        One of the main ideas in our arguments of Section~\ref{sec: Integral Geometry Formula} is to locally approximate the smooth sublocus of a stratified $p$-adic algebraic set by its tangent cone. In this section, we review the technical results needed to implement this idea.
        

        
        \begin{lemma}[Hensel's lemma]
            Let $\mathbf{f} = (f_1, \ldots, f_n) \in \Zp[x_1, \ldots, x_n]^n$, let $m$ be a positive integer, let $a \in \Zp^n$, and let $J_\mathbf{f}(a)$ be the Jacobian matrix of $\mathbf{f}$ at $a$. If
                $\norm{\mathbf{f}(a)} < \abs{J_{\mathbf{f}}(a)}^2$, then there is a unique $\alpha \in \Zp^n$ such that $\mathbf{f}(\alpha) = 0$ and $\alpha \equiv a \pmod p^m$.
        \end{lemma}
        
        \begin{proof}
            See Keith Conrad's notes \cite{Conrad} (unpublished) or \cite[Theorem 1]{Fisher} (published).
        \end{proof}
        
        The $p$-adic analytic implicit function theorem is the second critical tool we need. The formulation below is due to Igusa~\cite[Theorem 2.2.1]{Igusa}.

        \begin{theorem}[Implicit Function Theorem]
            Denote $\Zp \llbracket x,y \rrbracket = \Zp \llbracket x_1, \ldots, x_r, y_1, \ldots, y_{n-r} \rrbracket$ the ring of formal power series in the variables $x_1, \ldots, y_{n-r}$ with coefficients in $\Zp$. Then:
            \begin{enumerate}
                \item
                    If $F_i(x,y)$ is in $\Zp \llbracket x,y \rrbracket$ and $F_i(0,0) = 0$ for all $i$ and further
                        \[
                            \frac{\partial (F_1, \ldots, F_{n-r})}{ \partial (y_1, \ldots, y_{n-r})}(0,0) \not \equiv 0 \mod p,
                        \]
                    then every $f_i(x)$ in the unique solution $f = (f_1, \ldots, f_{n-r})$ of $F(x,f(x))=0$ satisfying $f_i(0)=0$ is in $\Zp \llbracket x \rrbracket$.
                \item
                    If $a \in B(0; p^{-1})$, then $f(a) \in \Zp^{n-r}$ and $F(a,f(a)) = 0$.
                    Furthermore, if $(a,b) \in p\Zp^r \times p\Zp^{n-r}$ satisfies $F(a,b) = 0$, then $b = f(a)$.
                    
            \end{enumerate}
        \end{theorem}

        The statement of part $(2)$ is slightly modified from Igusa~\cite[Theorem 2.2.1]{Igusa}, but is equivalent (see \cite[p22]{Igusa}). In our study of embedded $p$-adic algebraic sets, we can give an explicit quantitative version of the Implicit Function Theorem.

        \begin{proposition} \label{prop: isometric IFT}
            Let $X \subseteq \Qp^{n}$ be an algebraic set and let $a \in X$ be a smooth point. Let $\mathbf{f}$ be a set of $r$ local generators for $I(X) \cap \Zp[x_1, \ldots, x_n]$ at $a$. Let $m \in \mathbb{N}$ be such that $\abs{J_{\mathbf{f}}(a)}^2 \not \equiv 0 \pmod {p^m}$. Then there exists a bianalytic isometry $\pi\: (X \cap B(a; p^{-m})) \rightarrow (T_x X \cap B(a;p^{-m}))$, which in a suitable change of coordinates is given by a projection map.
        \end{proposition}
        
        \begin{proof}
            We may assume up to translation that $a$ is the origin. Note that we may also act on the equations by taking $\GL_r(\Zp)$-linear combinations and act by $\GL_n(\Zp)$-changes of coordinates on $X$. Thus, by using the Smith normal form of a matrix over $\Zp$, we may assume that
                \[
                    J_{\mathbf{f}}(0) = 
                    \left[
                    \begin{array}{ccc|ccc}
                        \sigma_1 & & 
                            & \ddots & \\
                        & \ddots & 
                            & & 0 \\
                        & & \sigma_r & 
                            & & \ddots
                    \end{array}
                    \right]
                \]
            with $\abs{\sigma_1} \geq \ldots \geq \abs{\sigma_r}$. We consider the new system of equations given by
                \[
                    F_j(x) := \sigma_j^{-1} \sigma^{-1} f_j( y )
                \]
            with $y = (  \sigma x_1, \ldots, \sigma x_n)$ and $\sigma := \sigma_1\ldots\sigma_r$. Observe:
            \begin{itemize}
                \item[--]
                    The constant term of each $F_j$ is zero.
                \item[--]
                    The linear term of $F_j$ is $x_j$, as the factors of $\sigma_j, \sigma$ in the linear term cancel.
                \item[--]
                    The higher order terms are integral, as $\abs{ \sigma_j \sigma } \geq \abs{\sigma}^2$ for all $j$.
            \end{itemize}
            Thus, each $F_j \in \Zp[x]$, and the Jacobian matrix for the equations $\{F_1, \ldots, F_r \}$ is
            $\left[ \begin{array}{c|c} I & 0 \end{array} \right]$.
            
            We now apply the Implicit Function Theorem; denoting $x_{\tail} := (x_{r+1} , \ldots, x_n)$, we obtain power series $g_1, \ldots, g_r \in \Zp \llbracket x_{r+1}, \ldots, x_n \rrbracket$ such that formally
                \[
                    F_j(g_1(x_\tail), \ldots, g_r(x_\tail), x_{r+1}, \ldots, x_n) = 0.
                \]
            Furthermore, on the ball $B(0;p^{-1})$ the $g_j$ are convergent and we have for all $b \in B(0;p^{-1})$ that $F(a,b) = 0$ implies $a = g(b)$. Finally, since the last $n-r$ coordinates are identified with the tangent space of $X$ at $0$, we have that each $g_j$ is at least quadratic in the $x_i$.
            
            Thus, each $g_j$ is Lipschitz on $B(0;p^{-1})$ with Lipschitz constant $1$. In particular, the map
                \[
                \begin{tabu}{rccc}
                    \gamma_F\: & B(0,p^{-1}) & \rightarrow & Z(F_1, \ldots, F_r)(\Zp) \\
                    & x & \mapsto & (g_1(x_\tail), \ldots, g_r(x_\tail), x_{r+1}, \ldots, x_n)
                \end{tabu}
                \]
            is an isometry, whose inverse is the projection.
            
            Since each $g_j$ is at least quadratic, we have that each $\sigma^{-1} g_j(x_\tail)$ is Lipschitz on $B(0; \abs{\sigma})$ with Lipschitz constant $1$. Denoting $g_j^\flat(x) := \sigma^{-1} g(x)$ we have the isometry
                \[
                \begin{tabu}{rccc}
                    \gamma_f\: & B(0, \abs{\sigma}) & \rightarrow & Z(f_1, \ldots, f_r)(\Zp) \\
                    & x & \mapsto & (g_1^\flat(x_\tail), \ldots, g_r^\flat(x_\tail), x_{r+1}, \ldots, x_n)
                \end{tabu}
                \]
            whose inverse is the projection. Finally, since $p^{-m} \leq \abs{J_\mathbf{f}(0)}^2 \leq \abs{\sigma}$, we are done.
        \end{proof}

        Note that Proposition~\ref{prop: isometric IFT} holds for any choice of local equations for the algebraic set $X$. It is useful to fix some global set of defining equations with nice properties. For instance, we use the following elementary lemma.
        
        \begin{lemma}
            Let $X \ssq \P^n$ be an algebraic set, and let $I(X) \ssq \Qp[x_0, \ldots, x_n]$ be the homogeneous ideal of equations vanishing on $X$. Then $I(X) \cap \Zp[x_0, \ldots, x_n]$ is a homogeneous ideal of $\Zp[x_0, \ldots, x_n]$ saturated with respect to $\langle p \rangle$.
        \end{lemma}

        \begin{proof}
            If $f \in \Zp[x_0, \ldots, x_n]$ such that $p^m f \in I(X)$ for some $m \geq 0$, then $f$ vanishes at every $x \in X(\Qp)$, so by definition $f \in I(X)$. 
        \end{proof}


\section{Volumes} \label{sec: volumes}
\subsection{The volume of an algebraic set}

    \begin{definition}[The volume of an algebraic set]\label{def:voluma}
        Let $U\subset A\subseteq \Zp^n$ be an open and compact subset of an algebraic set $A$ of dimension $a$. We define the $a$-dimensional volume of $U$ as:
        \be\label{eq:voldef} 
            \vol_a(U)=\lim_{m\to \infty}p^{m(n-a)}\cdot \mu_n\left(\bigcup_{x\in U}B(x, p^{-m})\right).
        \ee
        In a similar way, if $U\subseteq A\subseteq \P^n$ is an open and compact subset of an algebraic set $A$ of dimension $\dim(A)=a$, we define its $a$-dimensional projective volume by \eqref{eq:voldef}, replacing $\mu_n$ with the pushforward measure on $\P^n$ and taking balls with respect to the projective metric.
        Moreover, if $A\subseteq S^{n-1}\subset \Qp^n$, again we define the volume of $U\subseteq A$ by \eqref{eq:voldef}.
    \end{definition}


    For (open compact subsets of) algebraic sets, Serre \cite[pp148-149]{Serre} observed that the limit in Definition~\ref{def:voluma} exists and is finite. He moreover observed, via an inductive argument based on resolution of singularities, that the volume of an algebraic set is a rational number. The result also appears in an article by Oesterl\'e \cite[Theorem 2]{Oesterle}, though Oesterl\'e's proves that the limit exists by more elementary means. Oesterl\'e also proves the following:
    
    \begin{proposition}\label{proposition:Hausdorff}
        The volume of an open and compact subset $U\subseteq A$ of an algebraic set (both affine and projective) coincides with its $\dim(A)$-Hausdorff measure.
    \end{proposition}

    When $U\subseteq \Zp^n$ is an open set of full dimension we have $\textrm{vol}_n(U)=\mu_n(U)$. Additionally, the volume of $U\subseteq A\subseteq \P^n$ and the volume of the $(\dim(A)+1)$-dimensional set $\varphi^{-1}(U)\subseteq \varphi^{-1}(A)\subseteq S^{n}$ are related by:
        \be \label{eq:projvol}
            \vol_a(U)=\frac{\vol_{a+1}(\varphi^{-1}(U))}{\textrm{vol}_{1}(\Zp^*)}=\frac{\textrm{vol}_{a+1}(\varphi^{-1}(U))}{1-\frac{1}{p}}.
        \ee

\subsection{Volume and number of points in the modulo \texorpdfstring{$p^m$}{p\^{}m} reduction}

        \begin{lemma}\label{lemma:cover}
            We have the following two identities:
        \begin{enumerate}
            \setlength{\itemsep}{10pt}
            \item 
                Let $U\subseteq A\subseteq \Zp^n$ be an open and compact subset of an algebraic set $A$ of dimension $\dim(A)=a$. Then $N_m(U)$ is the minimum number of affine balls of radius $p^{-m}$ that we need to cover $U$.
            \item 
                Let $U\subseteq A\subseteq \P^n$ be an open and compact subset of an algebraic set $A$ of dimension $\dim(A)=a$. Then $N_m(U)$ is the minimum number of projective balls of radius $p^{-m}$ that we need to cover $U$. Moreover, we have that
                \[
                    N_m(U) = \frac{N_m(\varphi^{-1}(U))}{p^m\left(1-\frac{1}{p}\right)}
                \]
            \end{enumerate}
        \end{lemma}
        \begin{proof}Both results follow from the definition of $N_m(U)$ and Proposition \ref{prop: reduction commutes with projectivization}.
        \end{proof}
        
        \begin{corollary}\label{cor:volumepoint}We have the following two identities.
            \begin{enumerate}
                \item
                    Let $U\subseteq A\subseteq \Zp^n$ be an open and compact subset of an algebraic set $A$ of dimension $\dim(A)=a$. Then:
                    \be 
                        \mathrm{vol}(U)=\lim_{m\to \infty}\frac{N_m(U)}{p^{ma}}.
                    \ee
                \item
                    Let $U\subseteq A\subseteq \P^n$ be an open and compact subset of an algebraic set $A$ of dimension $\dim(A)=a$. Then:
                    \be 
                        \mathrm{vol}(U)=\lim_{m\to \infty}\frac{N_m(U)}{p^{ma}}.
                    \ee
            \end{enumerate}
        \end{corollary}
        
        \begin{proof}
            In the affine case, using Lemma~\ref{lemma:cover} we cover $U$ with $N_m(U)$ disjoint balls of radius $p^{-m}$ and centered at points $x_1, \ldots, x_{N_m(U)}\in U$. We have:
                \begin{alignat}{4}
                    \vol(U) &= \lim_{m\to \infty}p^{m(n-a)}\cdot \mu\left(\bigcup_{x\in U}B(x, p^{-m})\right) 
                    &&=\lim_{m\to \infty}p^{m(n-a)}\sum_{j=1}^{N_m(U)} \mu\left(B(x_j, p^{-m})\right) \\
                    &=\lim_{m\to \infty}\frac{N_m(U)}{p^{ma}}.
                \end{alignat}
            In the projective case, we use the result in the affine case to see that
                \begin{alignat}{4}
                    \vol(U)&=\frac{\vol(\varphi^{-1}(U))}{1-\frac{1}{p}} 
                    &=\lim_{m\to \infty}\frac{N_m(\varphi^{-1}(U))}{p^{m(a+1)}\left(1-\frac{1}{p}\right)}
                    &=\lim_{m\to \infty}\frac{N_m(U)}{p^{ma}}.  \tag*{\qedhere}
                \end{alignat}
        \end{proof}
            
        \begin{corollary}\label{coro:smooth}
            Let $\varphi^{-1}(A) = X\subset S^{n}$ be the pullback of an algebraic set of dimension $a$. Denote by $\mathrm{sm}(X)$ the set of analytically smooth points of $X$ in $\Zp^{n+1}$. Then
            $
                \mathrm{vol}(X)=\mathrm{vol}(\mathrm{sm}(X)).
            $
            \end{corollary}
    
        \begin{proof}
            Let $ Z $ be the singular locus of $A$. Since the dimension of $Z$ is strictly smaller than $\dim(X)=a+1$, we have
                \begin{align}
                    \lim_{m \rightarrow \infty}p^{-m(a+1)} N_m(X)& = \lim_{m \rightarrow \infty} p^{-m(a+1)}\left( N_m(\mathrm{sm}(X)) + N_m(\varphi^{-1}( Z)) \right) \\
                    &= \lim_{m \rightarrow \infty} p^{-m(a+1)}  N_m(\mathrm{sm}(X)). \tag*{\qedhere}
                \end{align}
        \end{proof}
        
        \subsection{Quantitative estimates and the Weil canonical measure}\label{sec:Hausdorff}
        
        For an embedded $p$-adic analytic set, we can enhance \cite[Theorem 2]{Oesterle} and give explicit quantitative estimates to compute the limit in Definition~\ref{cor:volumepoint}. Our explicit approach enables us to compute the volume of some specific algebraic sets in Section~\ref{sec: applications} (see Proposition \ref{propo:veronese} and Corollary \ref{cor:mahler}).
        
        \begin{proposition}\label{prop:stability}
            Let $X\subset \Zp^n$ be an algebraic set of dimension $d$ and $x\in X$ be a smooth point. Then there exists $m_x>0$ such that for all $\ell_1\geq \ell_2\geq m_x$ we have:
            \be 
                N_{\ell_1}(X\cap B(x; p^{-\ell_2}))=p^{(\ell_1-\ell_2)d}
            \ee
            Furthermore, if $\mathbf{f}$ is a set of local equations for $X$ at $x$, then any $m_x \geq -\log_p \abs{J_{\mathbf{f}}(x)}^2$ suffices.
        \end{proposition}
        
        \begin{proof}
            By Proposition~\ref{prop: isometric IFT} there is a bianalytic isometry
                \[
                    \pi\: X \cap B(x; p^{-m_x}) \rightarrow T_x X \cap B(x; p^{-m_x})
                \]
            where $m_x = -\log_p \abs{J_{\mathbf{f}}(x)}^2$. The result holds for the image of $\pi$, so the claim follows.
        \end{proof}
        
        \begin{corollary} \label{cor: point counts for integral smooth compact manifolds}
            Let $X\subset \Zp^n$ be a compact open set of a smooth algebraic set of dimension $d$, then there exists $m_0=m_0(X)>0$ such that for all $m\geq m_0$ we have
            \be N_m(X)=p^{(m-m_0)d}N_{m_0}(X).\ee
        \end{corollary}
        
        \begin{proof}
            Using Proposition \ref{prop:stability} and compactness of $X$, we see that there exists $m_0(X)$ such that for every $x\in X$ and every $\ell_2\geq\ell_1\geq m_0(X)$, we have 
            $ 
                N_{\ell_2}(B(x, p^{-\ell_1})\cap X)=p^{(\ell_2-\ell_1)d}.
            $
            Covering $X$ with $N_{m_0}(X)$ disjoint balls of the form $B(x_i, p^{-m_0})$, with each $x_i \in X$, we see
            \begin{align}
                N_m(X)&=\sum_{i=1}^{N_{m_0}(X)}\!\!\!N_m(B(x_i, p^{-m_0})\cap X)
                =\sum_{i=1}^{N_{m_0}(X)}\!\!\!p^{(m-m_0)d}=N_{m_0}(X) p^{(m-m_0)d}. \tag*{\qedhere}
            \end{align}
        \end{proof}
 
    
    As a first application of our quantitative estimates, we have:
    
    \begin{theorem}
        Let $\mathcal{X}$ be a subscheme of $ \P^n$  which is smooth over $\Spec \Zp$. Then the Weil canonical volume of $X=\mathcal{X}(\Zp)$ is equal to the volume as defined above.  
    \end{theorem}
    
    \begin{proof}
       The Weil canonical volume of $X$ is $\frac{\#X(\mathbb{F}_p)}{p}$ \cite[Theorem 2.2.5]{WeilAndre}. However, we also have that the Jacobian matrix of the defining equations is non-zero modulo $p$ at any point in $X$, so the limit in Corollary~\ref{cor:volumepoint} stabilizes in the first term. This proves the result.
    \end{proof}

\section{The \texorpdfstring{$p$}{p}-adic integral geometry formula} \label{sec: Integral Geometry Formula}
    
    \begin{lemma}
        Let $x \in S^{n}$ be a point, let $B \ssq S^{n}$ be a ball of radius $p^{-m}$ containing $x$, with $m >0$. Finally, let $U := \bigcup_{a \in \Zp^\times} aU$, where $\Zp^\times$ acts in the usual way. Then the subgroup
        \[
            \Fix(\GL _{n+1}(\Zp), U) := \{ g \in \GL_{n+1}(\Zp) : gU = U \}
        \]
        is a finite index subgroup of $\GL_{n+1}(\Zp)$ containing $\Zp^\times$.
    \end{lemma}

    \begin{proof}
        The action of $\GL_{n+1}(\Zp)$ on $S^{n}$ commutes with reduction to $R_m$. The result is immediate since $\GL_{n+1}(R_m)$ is a finite group.
    \end{proof}


    The following technical lemma is simply the result that if $X,Y$ are random linear subspaces of complementary codimension, the probability densities
        \[
            P( \{x\}= X \cap Y \mid x \in X), \qquad P(\{y\} = X \cap Y \mid y \in Y)
        \]
    are uniform on $X$ and $Y$ (resp.), and furthermore are independent.

    \begin{lemma} \label{lemma:linear1}
        Let $  X,   Y,   H \ssq \P^{n}$ be linear varieties such that $\codim   X + \codim   Y + \codim   H = n$. Let $U_x \ssq   X, U_y \ssq   Y$ be relatively open balls (i.e, open balls of the subspaces $X,Y$). Then 
        \[
            \int_{\GL_{n+1}(\Zp)} \int_{\GL_{n+1}(\Zp)} \#( g_x  U_x \cap g_y  U_y \cap   H ) dg_y dg_x = 
            \frac{\vol(   U_x)}{\vol(  X)} \cdot
            \frac{\vol(   U_y)}{\vol(  Y)}.
        \]
    \end{lemma}

    \begin{proof}
        The strategy of the proof is to replace the integrand with a constant function by using the fact that $X$ (resp. $Y$) is the disjoint union of finitely many copies of $U_x$ (resp. $U_y$).
        
        Let $G_X := \Fix( \GL_{n+1}(\Zp),   X)$ and let $G_Y = \Fix( \GL_{n+1}(\Zp),   Y)$. Furthermore, let $A := \Fix(G_X,   U_x)$ and $B := \Fix(G_Y,   U_y)$. Note that $G_X$ acts on $  X$ through $\GL_{\dim   X + 1}(\Zp)$, so we have the commutative diagram
        \be 
            \begin{tikzcd}
            G_X \arrow[r]  & \GL_{\dim  {X}+1}(\mathbb{Z}_p)                        \\
            {\textrm{Fix}(G_X, U_x)} \arrow[r] \arrow[u, hook] & {\textrm{Fix}(\textrm{GL}_{\dim X+1}(\mathbb{Z}_p), U_x)} \arrow[u, hook]
            \end{tikzcd}.
        \ee
        In particular, $A$ has finite index in $G_X$ by the previous lemma. Denote $a := [G_X : A]$, $b := [G_Y : B]$ and let $g_{x,1}, \ldots, g_{x,a}$ and $g_{y,1}, \ldots, g_{y,b}$ be coset representatives for $A \ssq G_X$ and $B \ssq G_Y$ respectively.
        
        Note that $  X = \coprod_{j=1}^a g_{x,j}   U_x$. As $G_X$ acts transitively on the balls in $  H$ of a fixed radius, we have for any $i,j$ that
            \[
                \int_{G_X} \#( g_x  U_x \cap g_y  U_y \cap   H ) dg_y dg_x = \int_{G_X} \#( g_x g_{x,j}  U_x \cap g_y g_{y,i}  U_y \cap   H ) dg_y dg_x.
            \]
        In particular, if $\# (  X \cap   Y \cap   H) = 1$, we have
            \begin{align*}
                 ab\int_{G_X} \int_{G_Y} \#( g_x  U_x \cap g_y  U_y \cap   H ) dg_y dg_x
                &=
                \sum_{j=1}^a \sum_{i=1}^b \int_{G_X} \int_{G_Y} \#( g_x g_{x,j}  U_x \cap g_y g_{y,i}   U_y \cap   H ) dg_y dg_x
                \\
                &= \int_{G_X} \int_{G_Y} \#( g_x  {X} \cap g_y  Y \cap   H ) dg_y dg_x
                \\
                &=  \int_{G_X} \int_{G_Y}  \#(  {X} \cap   Y \cap   H ) dg_y dg_x
                \\
                &= \int_{G_X} \int_{G_Y} dg_y dg_x
            \end{align*}
       We now write $\GL_{n+1}(\Zp) \cong G_X \times (G_X \bs \GL_{n+1}(\Zp))$ as a product measure space, with corresponding differential $dg_x = ds_x dt_x$, using the orbit-stabilizer theorem. Thus, by Fubini's theorem we have
       \begin{align*}
        &\phantom{=} \ ab \int_{\GL_{n+1}(\Zp)} \int_{\GL_{n+1}(\Zp)} \#( g_x  U_x \cap g_y  U_y \cap   H ) dg_y dg_x
        \\ &=
        ab \int_{G_X \bs \GL_{n+1}(\Zp)} \int_{G_X} \int_{G_Y \bs \GL_{n+1}(\Zp)} \int_{G_Y} \#( g_x  U_x \cap g_y  U_y \cap   H ) ds_y dt_y ds_x dt_x
        \\ &=
        ab \int_{G_X \bs \GL_{n+1}(\Zp)} \int_{G_Y \bs \GL_{n+1}(\Zp)} \int_{G_X} \int_{G_Y} \#( g_x  U_x \cap g_y  U_y \cap   H ) ds_y ds_x dt_y dt_x
        \\ &=
        \int_{G_X \bs \GL_{n+1}(\Zp)} \int_{G_Y \bs \GL_{n+1}(\Zp)} \left( ab \int_{G_X} \int_{G_Y} \#( g_x  U_x \cap g_y  U_y \cap   H ) ds_y ds_x \right) dt_y dt_x
       \end{align*}
        Since the set of $(g_x, g_y)$ in $\GL_{n+1}(\Zp)$ where $g_x   X,  g_y   Y,   H$ do not meet transversely has measure zero, we have that our expression above equals
        \begin{align*}
            \int_{G_X \bs \GL_{n+1}(\Zp)} \int_{G_Y \bs \GL_{n+1}(\Zp)} \int_{G_X} \int_{G_Y} ds_y ds_x dt_y dt_x = 1.
       \end{align*}
        Finally, since the disjoint union of $a$ copies of $  U_x$ is $  X$, we have that
            $
                \frac{1}{a} = \frac{\vol(  U_x)}{\vol(  X)} 
            $
    \end{proof}
 
    
    \begin{corollary}
        Let $  X,   H \ssq \P^{n}$ be linear varieties such that $\codim   X + \codim   H = n$. Let $  U_x \ssq   X$ be a relatively open ball. Then 
        \[
            \int_{\GL_{n+1}(\Zp)} \#( g_x  U_x \cap   H ) dg_x = 
            \frac{\vol(   U_x)}{\vol(  X)}.
        \]
    \end{corollary}

    \begin{proof}
        Let $  Y := \P^n$ and $  U_y :=   Y$. By Lemma~\ref{eq:linear1} we have that 
        \[
            \int_{\GL_{n+1}(\Zp)} \int_{\GL_{n+1}(\Zp)} \#( g_x  U_x \cap g_y  U_y \cap   H ) dg_y dg_x = 
            \frac{\vol(   U_x)}{\vol(  X)} \cdot
            \frac{\vol(   U_y)}{\vol(  Y)} = \frac{\vol(   U_x)}{\vol(  X)}.
        \]
        But we see that the inner integral is constant in $g_y$ and equal to $\#(  U_x \cap   H)$ since $g_y   U_y = \P^n$.
    \end{proof}
 
    We define some notation to state the next lemma. Let $\mathbf{f}_1, \mathbf{f}_2 \ssq \Zp[x_0,x_1, \ldots, x_n]$ be sets of homogeneous polynomials of size $n-r$, $n$ respectively, and let $x,y \in \P^{n}$. We define 
        \[
            J\left(\mathbf{f}_1(x), \mathbf{f}_2(y) \right) := 
            \begin{bmatrix}
                J_{\mathbf{f}_1}(x) \\
                J_{\mathbf{f}_2}(y)
            \end{bmatrix}
            \qquad \text{and} \qquad
            \abs{J\left(\mathbf{f}_1(x), \mathbf{f}_2(y) \right)} =
            \frac
            {\abs{\det J\left(\mathbf{f}_1(x), \mathbf{f}_2(y) \right) } }
            {\norm{x}^{\delta_1} \norm{y}^{\delta_2}}
        \]
    where $\delta_j = \sum_{f \in \mathbf{f}_j} \deg f$.
    Additionally, if $g \in \GL_{n+1}(\Zp)$ and $\mathbf{f} = (f_1, \ldots, f_{n-r})$, we denote $\mathbf{f}^g := (f_1 \circ g, \ldots, f_{n-r}\circ g)$.

    \begin{lemma}[Linear Approximation Lemma] \label{lemma: linear approximation}
        Let $X,Y \subseteq \P^{n}$ be algebraic sets of complementary codimension. Let $x \in X, y \in Y$ be smooth points, contained in open balls $U_x, U_y$ of $\P^n$ of radius $p^{-m}$, with local equations $\mathbf{f}_x, \mathbf{f}_y$ (respectively). If
            \[
                \abs{J\left(\mathbf{f}_x(x), \mathbf{f}_y(y) \right) }^2 > p^{-m}
                \qquad \text{or} \qquad
                U_x \cap U_y = \emptyset,
            \]
        then
            $
                \#({X} \cap {U_x} \cap {Y} \cap {U_y})
                = \#(T_{x}  X \cap   U_x \cap T_{y}  Y \cap U_y).
            $
    \end{lemma}

    \begin{proof}
        If $U_x \cap U_y = \emptyset$, we have $\#({X} \cap {U_x} \cap {Y} \cap {U_y})
                = \#(T_{x}  X \cap   U_x \cap T_{y}  Y \cap U_y) = 0$.
        
        Otherwise, note that $U_x,U_y$ are balls of radius $p^{-m}$ with nontrivial intersection, so $U_x = U_y$. In particular, $\pi_m(x) = \pi_m(y)$. Choosing representatives on $S^n$ for $x,y$, we have
        \[
            J\left(\mathbf{f}_x(x), \mathbf{f}_y(y) \right) \equiv J\left(\mathbf{f}_x(x), \mathbf{f}_y(x) \right) \pmod {p^{-m}}.
        \]
        The right hand side is just the Jacobian matrix for $\mathbf{f}_x \cup \mathbf{f}_y$ at $x$. By Hensel's lemma, there is a unique projective solution to the system $\mathbf{f}_x \cup \mathbf{f}_y$ which is contained in $U_x = U_y$. On the other hand, the equations defining the two tangent spaces $T_x X, T_y X$ also satisfy the conditions of Hensel's lemma, and have the common interesction point $x \equiv y \pmod {p^{-m}}$. Thus, the conclusion of the lemma holds in this case as well.
    \end{proof}
    
    \begin{theorem}\label{thm:volint}
        Let $A\subset S^{n}$ be a homogeneous algebraic set of dimension $\dim(A)=a+1$ and let $ {X}\subset \P^{n}$ be its projectivization. If $ {U}\subseteq  {X}$ is a compact open subset and $ L\subset \P^{n}$ is a projective subspace of dimension $n-a$, then for almost all $g\in \mathrm{GL}_{n+1}(\Zp)$ the intersection $ {U}\cap g L$ is transversal and:
        \be \mathrm{vol}(\varphi^{-1}( {U}))=\mathrm{vol}(\P^{a})\cdot \int_{\mathrm{GL}_{n+1}(\Zp)}\# \left( {U}\cap g L \right) dg.\ee
    \end{theorem}
    
    \begin{proof}
    Let $Z\subset \textrm{GL}_{n+1}(\Zp)$ be the set defined by:
    \be Z=\{g : \textrm{$g {L}$ is not transversal to $ {X}$}\}.\ee
    and denote $Z_\ell := \bigcup_{g\in Z}B(g, p^{-\ell})$. Since $Z$ is contained in a proper algebraic set, by Corollary~\ref{cor:volumepoint} there is a constant $C>0$ such that
    \be \mu\left(\bigcup_{g\in Z}B(g, p^{-\ell})\right)\leq Cp^{-\ell \dim(Z)}.\ee
    With $\ell$ fixed, we have for every $u\in U$ there exists $m_{u, \ell}>0$ such that for all $g\in \textrm{GL}_{n+1}(\Zp) \bs Z_\ell$ and for every $m\geq m_{u, \ell}:$
        \be\label{eq:linear1} 
            \#\left(g {L}\cap  {U}\cap  B(u, p^{-m}))\right)=\#\left(g {L}\cap T_{u} {U}\cap  B(u, p^{-m}))\right).
        \ee
    \noindent
    Observe that $Z_1 \supset Z_2 \supset \ldots $ and that $\bigcap_{\ell \in \mathbb{N}} Z_\ell = Z$. Since $Z$ has measure zero, we see from the Monotone Convergence Theorem that
        \be 
            \int_{\textrm{GL}_{n+1}(\Zp)}f(g)dg = \lim_{\ell\to\infty}\int_{\mathrm{GL}_{n+1}(\Zp)\backslash Z_\ell}f(g)dg
        \ee
    for every measurable function $f:\textrm{GL}_{n+1}(\Zp)\to\mathbb{R}$. With this in mind, we evaluate now:
    \begin{align}
    \int_{\mathrm{GL}_{n+1}(\Zp)}\# \left( {U}\cap g L \right) dg&=\lim_{\ell\to \infty}\int_{\mathrm{GL}_{n+1}(\Zp)\backslash Z_\ell}\# \left( {U}\cap g L \right) dg\\
    &=\lim_{\ell\to \infty}\int_{\mathrm{GL}_{n+1}(\Zp)\backslash Z_\ell}\sum_{i=1}^{N_m(U)}\#\left( {U}_i\cap g  {L}\right)dg=(*),
    \end{align}
    where we have covered $ {U}=\bigcup_{i=1}^{N_m(U)}\left( B(u_i, p^{-m}) \cap  {U}\right)$ with $N_m(U)$ disjoint subsets of the form $U_i=B(u_i, p^{-m}) \cap  {U}$, with
    \be m=m_\ell\geq\sup_{u\in  {U}}m_{u, \ell}\geq 0\ee large enough (such an $m$ exists by compactness of $ {U}$). 
    
    Using \eqref{eq:linear1}, we can continue with:
    \begin{align}
        (*)&=\lim_{\ell\to \infty} 
        \bigint_{\!\!\!\!\mathrm{GL}_{n+1}(\Zp)\backslash Z_\ell} \sum_{i=1}^{N_m(U)}\#\left( B(u_i, p^{-m})\cap T_{u_i} {U}_i\cap g  {L}\right)dg 
        \\[1ex]
        &=\lim_{\ell\to \infty} 
        \left(
            \bigint_{\!\!\!\! \mathrm{GL}_{n+1}(\Zp)} 
            \sum_{i=1}^{N_m(U)}
                \#\left( B(u_i, p^{-m})\cap T_{u_i} {U}_i\cap g {L} \right) dg
            +O(p^{-\ell\dim(Z)})
        \right) \\
        &=\lim_{\ell\to \infty} \bigint_{\!\!\!\! \mathrm{GL}_{n+1}(\Zp)}                 \sum_{i=1}^{N_m(U)}
                \#\left( B(u_i, p^{-m})\cap T_{u_i} {U}_i\cap g  {L}\right)dg
            \\[1ex]
        &=\lim_{\ell\to \infty}
        N_m(U) \cdot \frac{\vol_{a}( B(u, p^{-m})\cap \P^{a})}{\vol_{a}(\P^{ a})}\\
        &=\lim_{m\to \infty} N_m(U) \cdot \frac{\vol_{a}( B(u, p^{-m})\cap \P^{a})}{\vol_{a}(\P^{ a})} 
        \quad =\frac{\vol(U)}{\vol(\P^{a})} \qquad \textrm{by Corollary \ref{cor:volumepoint}}.
        \end{align}
        The conclusion follows now from the definition \eqref{eq:projvol}.
    \end{proof}
    
    \begin{theorem}[The $p$-adic Integral Geometry Formula]\label{thm:IGF}
        Let $ {U},  {V}\subset \P^{n} $ be open and compact subsets of algebraic sets of dimensions $\dim( {U})=a, \dim( {V})=b.$ Then for almost all $g\in \GL _{n+1}(\Zp)$ the intersection $ {U}\cap g {V}$ is transversal, of dimension $c=n-(n-a)-(n-b)$, and:
        \be \int_{\mathrm{GL}_{n+1}(\Zp)}\frac{\mathrm{vol}\left( {U}\cap g {V}\right)}{\mathrm{vol}\left(\P^c\right)} dg=\frac{\mathrm{vol}( {U})}{\mathrm{vol}(\P^a)}\cdot \frac{\mathrm{vol}( {V})}{\mathrm{vol}(\P^b)}.\ee
    \end{theorem}
    \begin{proof}The proof proceeds similarly to the previous one. We observe first that, whenever the intersection $  U\cap g   V$ is transversal, picking $ {L}\simeq \P^c\subset \P^{n}$ and using Theorem \ref{thm:volint}, we can write:
    \be
    \frac{\mathrm{vol}\left( {U}\cap g {V}\right)}{\mathrm{vol}\left(\P^c\right)}=\int_{\mathrm{GL}_{n+1}(\Zp)}\# \left( {U}\cap g {V}\cap g_2  L\right)dg_2,\ee
    and consequently:
    \be\label{eq:red} \int_{\mathrm{GL}_{n+1}(\Zp)}\frac{\mathrm{vol}\left( {U}\cap g {V}\right)}{\mathrm{vol}\left(\P^c\right)} dg=\int_{\mathrm{GL}_{n+1}(\Zp)}\int_{\mathrm{GL}_{n+1}(\Zp)}  \# \left( {U}\cap g_1 {V}\cap g_2  L\right)dg_2dg_1.\ee
    Denote by $G=\mathrm{GL}_{n+1}(\Zp)$ and consider the set:
    \be Z=\{(g_1, g_2)\in G\times G : \textrm{$  {U}\cap g_1 {V}\cap g_2 {L}$ is not transversal}\}.
    \ee
    Given $\ell>0$, for every $ y=(u, v)\in  {U}\times  {V}$ there exists $m_{y, \ell}>0$ such that for all $(g_1, g_2)\in G\times G\backslash \left(\bigcup_{z\in Z}B(z, p^{-\ell})\right)$ and for all $m\geq m_{y, \ell}$ we have:
        \be
        \begin{tabu}{rcccccl}
            & \#\big( &  U  \cap   B(u, p^{-m})  &\cap& 
            g_1 \left(  V \cap B(v, p^{-m}) \right) &\cap& 
            g_2  L  \big) \\[1ex]
        = & \#\big( & T_u  U\cap   B(u, p^{-m}) &\cap& 
            g_1 \left(T_v  V \cap   B(v, p^{-m}) \right) &\cap& 
            g_2  L \big).
        \end{tabu}
        \ee
    As in the proof of Theorem~\ref{thm:volint}, since $Z$ is contained in a proper algebraic set, we also have:
    \be \mu\left(\bigcup_{z\in Z}B(z, p^{-\ell})\right)\leq O(p^{-\ell}).\ee
   As before, we denote $Z_\ell=\bigcup_{z\in Z}B(z, p^{-\ell})$.
   
   By compactness of $U \times V$, we choose some $m=m_\ell\geq\sup_{y\in U \times V}m_{y, \ell}> 0$. 
   We cover now 
   \be 
    {U}=\bigcup_{i=1}^{N_m(U)}\left( B(u_i, p^{-m}) \cap  {U}\right), \qquad {V}=\bigcup_{i=1}^{N_m(V)}\left( B(v_i, p^{-m}) \cap  {V}\right)
   \ee
   with $N_m(U)$ (resp. $N_m(V)$) disjoint projective balls of radius $p^{-m}$. In the sequel we denote $  U_{i, m}= B(u_i, p^{-m}) \cap  {U}$ and $  V_{i, m}= B(v_i, p^{-m}) \cap  {V};$ we will also write $T  U_{i, m}=  B(u_i, p^{-m}) \cap T_{u_i} {U}$ and $T  V_{j, m}=  B(v_j, p^{-m}) \cap T_{v_j} {V}$.
   
    We proceed now to evaluate the integral in the statement, using \eqref{eq:red}:
    \begin{align}\int_{\mathrm{GL}_{n+1}(\Zp)}\frac{\mathrm{vol}\left( {U}\cap g {V}\right)}{\mathrm{vol}\left(\P^c\right)} dg&=\int_{\mathrm{GL}_{n+1}(\Zp)}\int_{\mathrm{GL}_{n+1}(\Zp)}  \# \left( {U}\cap g_1 {V}\cap g_2  L\right)dg_2dg_1\\
    &=\lim_{\ell\to \infty}\int_{G\times G\backslash Z_\ell} \# \left( {U}\cap g_1 {V}\cap g_2  L\right)dg_2dg_1\\
    &=\lim_{\ell\to \infty}\int_{G\times G\backslash Z_\ell} \sum_{i=1}^{N_m(  U)}\sum_{j=1}^{N_m(  V)}\# \left( {U}_{i, m}\cap g_1 {V}_{j, m}\cap g_2  L\right)dg_2dg_1\\
    &=\lim_{\ell\to \infty}\int_{G\times G\backslash Z_\ell} \sum_{i=1}^{N_m(  U)}\sum_{j=1}^{N_m(  V)}\# \left(T {U}_{i, m}\cap g_1T {V}_{j, m}\cap g_2  L\right)dg_2dg_1\\
    &=\lim_{\ell\to \infty}\left(\int_{G\times G} \sum_{i=1}^{N_m(  U)}\sum_{j=1}^{N_m(  V)}\# \left(T {U}_{i, m}\cap g_1T {V}_{j, m}\cap g_2  L\right)dg_2dg_1+O(p^{-\ell})\right)\\
    &=\lim_{\ell\to \infty}\int_{G\times G} \sum_{i=1}^{N_m(  U)}\sum_{j=1}^{N_m(  V)}\# \left(T {U}_{i, m}\cap g_1T {V}_{j, m}\cap g_2  L\right)dg_2dg_1=(*)
    \end{align}
    We use now Lemma \ref{lemma:linear1} and continue with:
    \begin{align}(*)&=\lim_{\ell\to \infty}N_m(  U)\frac{\textrm{vol}_{a}( B(u, p^{-m})\cap \P^{a})}{\textrm{vol}_{a}(\P^{a})}\cdot N_m(  V)\frac{\textrm{vol}_{b}( B(v, p^{-m})\cap \P^{b})}{\textrm{vol}_{b}(\P^{b})}\\
    &=\lim_{m\to \infty}N_m(  U)\frac{\textrm{vol}_{a}( B(u, p^{-m})\cap \P^{a})}{\textrm{vol}_{a}(\P^{a})}\cdot N_m(  V)\frac{\textrm{vol}_{b}( B(v, p^{-m})\cap \P^{b})}{\textrm{vol}_{b}(\P^{b})}\\
   &= \frac{\textrm{vol}_{a}(  U)}{\textrm{vol}_{a}(\P^{a})}\cdot \frac{\textrm{vol}_{b}( {V})}{\textrm{vol}_{b}(\P^{b})}. \tag*{\qedhere}
    \end{align}
    \end{proof}

\section{How many zeroes of a random \texorpdfstring{$p$}{p}-adic polynomial are in \texorpdfstring{$\Zp$}{Zp}?}\label{sec: applications}
\begin{theorem}\label{thm:randompoly}Let $\{f_0, \ldots, f_\ell\}$ be a basis for $\Qp[x_0, \ldots, x_n]_{(d)}$ and for every $i=1,\ldots, n$ consider the random polynomial:
\be F_i(x)=\sum_{k=0}^{\ell}\xi_{i, k}f_{k}(x)\ee
where $\{\xi_{i,k}\}_{i=1, \ldots, n, k=0, \ldots, \ell}$ is a family of independent uniformly distributed random variables in $\Zp$. Let $\nu(\P^n)\subset \P^{\ell}$ be the image of the ``Veronese map'' $\nu:\P^n\to \P^{\ell}$
\be\label{eq:veronese} \nu([x])=[f_0(x), \ldots, f_\ell(x)].\ee
Then, the expectation of the number of solution of the random system of equations $F_1=\cdots=F_n=0$ in $\P^n$ equals:
\be \mathbb{E}\#\{F_1=\cdots=F_n=0\}=\frac{\mathrm{vol}(\nu(\P^n))}{\mathrm{vol}(\P^n)}.\ee
\end{theorem}
\begin{proof}The result is an easy consequence of the Integral Geometry Formula. We use the fact that $\{f_0, \ldots, f_{\ell}\}$ is a linearly independent and spanning set to ensure that the map $\nu$ is an embedding and that $\nu(\P^n)$ is a smooth algebraic subset. Now for every $i=1, \ldots, n$ we can consider the random hyperplane $L_i\subset \P^\ell$:
\be L_i=\left\{\sum_{k=0}^\ell \xi_{i, k} y_i=0\right\}.\ee
The distribution of this hyperplane is invariant under the group $\GL_{\ell+1}(\Zp)$ and can be alternatively described as fixing a hyperplane $L_0\subset \P^\ell$, sampling $g_i\in \GL_{\ell+1}(\Zp)$ uniformly and setting $L_i=g_i L_0$. Consequently:
\begin{align}\mathbb{E}\#\{F_1=\cdots=F_n=0\}&=\mathbb{E}\#\nu(\P^n)\cap L_1\cdots\cap L_{n}\\
&=\mathbb{E}\#\nu(\P^n)\cap g_1L_0\cdots\cap g_nL_{n}\\
&=\mathbb{E}_{g_1, \ldots, g_{n-1}}\left(\mathbb{E}_{g_n}\#\nu(\P^n)\cap g_1L_0\cdots\cap g_nL_{0}\right)\\
&=\mathbb{E}_{g_1, \ldots, g_{n-1}}\left(\int_{\GL_{\ell+1}(\Zp)}\#\nu(\P^n)\cap g_1 L_0\cdots\cap g_nL_{0} \ dg_n\right)\\
&=\mathbb{E}_{g_1, \ldots, g_{n-1}}\left(\frac{\mathrm{vol}(\nu(\P^n)\cap g_1L_0\cdots\cap g_{n-1}L_{0})}{\mathrm{vol}(\P^1)}\right)
=(*),
\end{align}
where in the last step we have used the Integral Geometry Formula from Theorem \ref{thm:IGF}.
Repeating this process iteratively we get:
\begin{align}(*)&=\mathbb{E}_{g_1, \ldots, g_{n-2}}\left(\frac{\mathrm{vol}(\nu(\P^n)\cap g_1L_0\cdots\cap g_{n-1}L_{0})}{\mathrm{vol}(\P^2)}\right)\\
&=\cdots\\
&=\mathbb{E}_{g_1}\left(\frac{\mathrm{vol}(\nu(\P^n)\cap g_1 L_0)}{\mathrm{vol}(\P^{n-1})}\right)
=\frac{\mathrm{vol}(\nu(\P^n))}{\mathrm{vol}(\P^{n})}. \tag*{\qedhere}
\end{align}
\end{proof}
The previous theorem reduces the question of the expected number of zeros of a random polynomial systerm to computing the volume of the ``Veronese variety'' $\nu(\P^n)\subset \P^{\ell}$ (i.e. the image of $\nu$ from \eqref{eq:veronese}).

In order to actually calculate the volume of the image of the Veronese, it is tremendously convienient to have available a coarea formula. Specifically, a statement which allows us to calculate the $1$-dimensional volume of a curve in $S^d$.

\begin{lemma}[Arc length formula] \label{lemma:volcurve}
     Let $U\subset \Qp$ be an open and compact set and $\gamma:U\to \Qp^n$ be an analytic embedding such that for all $u\in U$ we have  $| J\gamma(u)|=p^{b}$ for some $b\in \mathbb{Z}$, then $\mathrm{vol}_1(\gamma(U))=p^{b}\mathrm{vol}_1(U).$
\end{lemma}

\begin{proof} 
    Using Taylor's formula for every $u\in U$ and $\epsilon\in \Qp$ sufficiently small, we have:
        \be
            \gamma(u+\epsilon)-\gamma(u)=J\gamma(u)\epsilon+O(\norm{\epsilon}^2).
        \ee
    Because $\Qp^n$ is an ultrametric space, this implies that for $\epsilon$ small enough:
        \be 
            \norm{\gamma(u+\epsilon)-\gamma(u)}=|J\gamma(u)|_p \norm{\epsilon}.
        \ee
    In particular, denoting $|J\gamma(u)|=p^b$, for every $u\in U$ there is $m_u\in \mathbb{N}$ such that for all $m\geq m_0$ 
        \be\label{eq:bbb} 
            \gamma(B(u, p^{-m}))=\gamma(U)\cap B(\gamma(u), p^{b-m}).
        \ee
    Since $U$ is assumed to be compact, then there exists $m_0>0$ such that for all $m\geq m_0$ equation \eqref{eq:bbb} is true for all $u\in U$.
    Now, covering $U$ with $N_m$ disjoint balls \be B(u_{m,1}, p^{-m}), \ldots, B(u_{m,N_m}, p^{-m}),\ee  we can write $\mathrm{vol}_1(U)=p^{-m}{N_m}$ for $m$ large enough:

    At the same time the balls $B(\gamma(u_{m,1}), p^{b-m}), \ldots, B(\gamma(u_{m,N_m}), p^{b-m})$ cover $\gamma(U)$ and for large enough $m$ are disjoint. In particular,
        \begin{align} 
            p^b\mathrm{vol}_1(U)&=p^{b}\lim_{m\to \infty}\frac{N_m}{p^m}
            =\lim_{m\to \infty} \frac{N_{m-b}(\gamma(U))}{p^{m-b}}
            =\lim_{m'=m-b\to \infty} \frac{N_{m'}(\gamma(U))}{p^{m'}}
            =\mathrm{vol}(\gamma(U)). \ \ \tag*{\qedhere}
        \end{align}
\end{proof}


\subsection{The standard Veronese random model}

\begin{proposition}\label{propo:veronese}Let $\P^{\ell}$ be the projectivization of the space of homogeneous polynomials with coefficients in $\Qp$, of degree $d$ and in $n+1$ variables (so that $\ell={\binom{n+d}{d}}-1$). Let $\nu:\P^n\to \P^{\ell}$ be the Veronese map
\be\nu_{n, d}:[x_0, \ldots, x_n]\mapsto [(x_0^{\alpha_0}\cdots x_n^{\alpha_n})_{|\alpha|=d}].
\ee
Then $\mathrm{vol}(\nu(\P^n))=\mathrm{vol}(\P^n).$ 
\end{proposition}
\begin{proof}We will prove that $\nu$ is an isometry onto its image: from this the result will follow, since Proposition \ref{proposition:Hausdorff} establishes that the volume in our sense is a metric invariant.

Observe first that the linear representation $\rho:\GL_{n+1}(\Zp)\to \GL_{\ell}(\Qp)$ by change of variables has image in $\GL_{\ell}(\Zp)$, hence for every $g\in \GL_{n+1}(\Zp)$, the map $\rho(g):\P^{\ell}\to \P^\ell$ acts by isometries (as well as the map $g:\P^n\to \P^n$).

Let now $x, y\in \P^n$ be any two points. Pick an element $g\in \GL_{n+1}(\Zp)$ such that:
\be g[x]=[1,0, \ldots, 0]\quad \textrm{and}\quad g[y]=[a_0, a_1, 0, \ldots, 0]\ee
with $(a_0, a_1, 0, \ldots,0)\in S^n$. Then, because $g$ is an isometry, using the definition \eqref{eq:distproj} of the distance in projective spaces:
\be\label{eq:part} d_{\P^n}([x], [y])=d_{\P^n}([1,0, \ldots, 0], [a_0, a_1, 0, \ldots, 0])=|a_1|_p.\ee
Let us evaluate now the distance $d([\nu(x)], [\nu(y)])$:
\begin{align}d_{\P^\ell}([\nu(x)], [\nu(y)])&=d_{\P^\ell}(\rho(g)[\nu(x)],\rho(g) [\nu(y)])\\
&=d_{\P^\ell}([\nu(gx)], [\nu(gy)])\\
&=d_{\P^\ell}([1, 0, \ldots, 0], \nu[a_0, a_1, 0, \ldots, 0]).
\end{align}
Since $(a_0, a_1, 0,\ldots, 0)$ was in the unit sphere, the vector $\nu(a_0, a_1, 0,\ldots, 0)\in \Qp^{\ell+1}$ is also on the unit sphere and its components are (in lexicographic order):
\be \nu(a_0, a_1, 0,\ldots, 0)=({a_0}^d, {a_0}^{d-1}a_1, \ldots, a_0{a_1}^{d-1}, {a_1}^{d}, 0, \ldots, 0).\ee
As a consequence, using again the definition \eqref{eq:distproj}, we see that $d_{\P^\ell}([1, 0, \ldots, 0], \nu[a_0, a_1, 0, \ldots, 0])$ equals the $p$-adic norm of the vector:
\begin{align} ({a_0}^{d-1}a_1, \ldots, a_0{a_1}^{d-1}, {a_1}^{d}, 0, \ldots, 0)&=a_1({a_0}^{d-1}, \ldots, a_0{a_1}^{d-2}, {a_1}^{d-1}, 0, \ldots)\\
&=a_1(\underbrace{\nu_{1, d-1}(a_0, a_1)}_{\in {\Qp}^{d}}, \underbrace{0, \ldots, 0}_{\in \Qp^{\ell-1-d}}).
\end{align}
But $\norm{\nu_{1, d-1}(a_0, a_1)} = 1$, so using \eqref{eq:part} we see:
\begin{align} d_{\P^\ell}([1, 0, \ldots, 0], \nu[a_0, a_1, 0, \ldots, 0])&= \norm{a_1(\nu_{1, d-1}(a_0, a_1), 0, \ldots, 0)}\\
&=\norm{ a_1(\nu_{1, d-1}(a_0, a_1)) }\\
&= d_{\P^n}([x], [y]) \tag*{\qedhere}
\end{align}
\end{proof}

\begin{corollary}\label{coro:veronese}
Let $\{\xi_{1,\alpha}\}_{i=1, \ldots, n, |\alpha|=d}$ be a family of independent uniformly distributed random variables in $\Zp$ and for every $i=1, \ldots, n$ define the random polynomial:
\be g_i(x)=\sum_{|\alpha|=d}\xi_{i, \alpha}x_0^{\alpha_0}\cdots x_n^{\alpha_n}.\ee
Then
$ \mathbb{E}\#\{g_1=\cdots=g_n=0\}=1. $
\end{corollary}
\begin{proof}
Proposition~\ref{propo:veronese} gives the volume of the standard Veronese. We apply Theorem~\ref{thm:randompoly}.
\end{proof}

\subsection{Evans' result and the Mahler Veronese}\label{sec:mahler}

The \emph{Mahler basis} for the space of (nonhomoegeneous) polynomials of degree $d$ in $n$ variables is defined as follows. For every $k_1,\ldots, k_n$ with $k_1+\cdots+k_n\leq d$ consider the polynomial:
    \begin{align} 
        T_{k_1, \ldots, k_n}(t_1, \ldots, t_n) &:={\binom{t_1}{k_1}}\cdots {\binom{t_n}{k_n}}
        =\frac{t_1\cdots (t_1-k_1+1)}{k_1!}\cdots \frac{t_n\cdots (t_n-k_n+1)}{k_n!}.
    \end{align}
The family $\{T_{k_1, \ldots, k_n}\}$ forms a basis for $\Qp[t_1, \ldots, t_n]_{d}$. Evans considers the random system of equations:
    \be\label{eq:evansp} 
        f_1=\ldots=f_n=0, \qquad \text{where} \qquad
        f_i(t_1, \ldots, t_n)=\sum_{k_1+\cdots k_n\leq d}\xi_{i, k}\cdot T_{k_1, \ldots k_n}(t_1, \ldots, t_n)
    \ee
and proves \cite[Theorem 1.2]{Evans} that the expected number of solutions of \eqref{eq:evansp} in $\Zp^n$ is
    \be 
        \mathbb{E}\#\{f_1=\ldots=f_n=0\}\cap \Zp^{n}=p^{n\lfloor\log_{p}d\rfloor}(1+p^{-1}+\cdots+p^{-n})^{-1}.
    \ee
We give a proof of Evans' result for the expected number of zeros in $\Zp$ using the integral geometry formula in the univariate case. We also extend Evans' result and compute the expected number of zeros in $\Qp \bs \Zp$. The argument for both results follow the same pattern. First, we will prove an elementary technical lemma regarding the size of binomial expressions. Next, we use the Arc length formula to compute the image of the appropriate Veronese map on the unit sphere and also the image in projective space. Finally, we apply the Integral Geometry formula to obtain results about the random model.

\begin{lemma}\label{lemma:normderivative}
    Let $a \in \Zp$ and let
        $
            F\: (t,1) \rightarrow \left( 1, \binom{t}{1}, \ldots, \binom{t}{d} \right)
        $
    be the affine Mahler Veronese map. Then the image of $F\: \Zp \rightarrow \Qp^{d+1}$ is contained in the unit sphere, and
        \[
            \abs{J_F(a)} = \max_{1 \leq k \leq d} \abs{k}^{-1} = p^{\lfloor \log_p d \rfloor}.
        \]
\end{lemma}

\begin{proof}
    Since $\mathbb{Z}$ is dense in $\Zp$ and $\Zp$ is a discrete valuation domain, we may assume that $a \in \mathbb{Z}$. The first claim is obvious since the first coordinate of $F(a)$ is a unit, and binomial coefficients are integers. Next, let
    \[
        g(t) := \frac{d}{dt} \binom{t}{k} = \sum_{j=0}^{k-1} \frac{1}{t-j} \binom{t}{k}
    \]
    and note each term of $g(t)$ is a polynomial. Let $\alpha$ be the largest (in the archimedean sense) $p$-th power in $\{1, \ldots, k\}$ and let $\beta \in \{a-k+1, \ldots, a\}$ be an element which is maximally $p$-divisible. Since both intervals of integers have length $k$, we have that $\abs{\beta} \leq \abs{\alpha}$. Furthermore, 
    \[
        \prod_{\substack{x = 1 \\ x \ne \alpha}}^k \abs{x} \geq 
        \prod_{\substack{y=a-k+1 \\ y \ne \beta}}^{a} \abs{y}.
    \]
    If follows for each $a \in \mathbb{Z}$ and $j \in \{0 , \ldots, k-1\}$ that
    \[
        \abs{\frac{1}{a-j} \binom{a}{k}} = \abs{ \frac{1}{\alpha}} 
        \left({ \displaystyle\prod_{\substack{y=a-k+1 \\ y \ne a-j}}^{a} \abs{y} } \right)
        \left( {\displaystyle \prod_{\substack{x = 1 \\ x \ne \alpha}}^k \abs{x}} \right)^{-1}
        \leq \abs{\alpha}^{-1}
    \]
    since the left hand side is maximal when ommiting the factor $a-j = \beta$ of minimal size. Thus, each term of the $k$-th entry of $J_F(a)$ is bounded by $\abs{\alpha}^{-1} \leq p^{\lfloor \log_p d \rfloor}$. 
    
    To demonstrate equality, we let $\alpha$ be the largest $p$-th power in $\{1, \ldots, d\}$ and we consider the $(\alpha + 1)$-th entry of $J_F(a)$. In this case, we have the equality
        \[
        \prod_{\substack{x = 1 \\ x \ne \alpha}}^\alpha \abs{x} = 
        \prod_{\substack{y=a-\alpha+1 \\ y \ne \beta}}^{a} \abs{y}.
        \]
    since both sets $\{1, \ldots, \alpha\}$, $\{a-\alpha+1, \ldots, a\}$ are a complete set of representatives of the integers modulo the $p$-th power $\alpha$. Thus,
            \be
            \frac{d}{dt} \binom{t}{\alpha}(a) = \sum_{j=0}^{\alpha-1} \frac{1}{a-j} \binom{a}{\alpha}
        \ee
    has a unique term of maximal size $\abs{\alpha}^{-1} = p^{\lfloor \log_p d \rfloor}$.
\end{proof}

\begin{corollary}\label{cor:mahler}
    The length of the image of the affine Mahler Veronese map $F\: (\Zp \times \{1\}) \rightarrow S^{d}$
    \[
            F\: (t,1) \rightarrow \left( 1, \binom{t}{1}, \ldots, \binom{t}{d} \right)
    \]
    is exactly $p^{\lfloor \log_p d \rfloor}$. Furthermore, with $\varphi\: S^d \rightarrow \P^d$, we have that
    \[
        \vol_1(F(\Zp)) = \vol_1(\varphi (F(\Zp))).
    \]
\end{corollary}

\begin{proof} 
  The first part of the statement, follows immediately from Lemma \ref{lemma:volcurve} (with the choice $U=\Zp$ and $\gamma(t)=F(1,t)$) and Lemma \ref{lemma:normderivative}:
  \be\mathrm{vol}_1{(F(\Zp))}=p^{\lfloor \log_p d\rfloor}\mathrm{vol}_1(\Zp)=p^{\lfloor \log_p d\rfloor}.
  \ee
  
    We now prove the second part of the claim. In fact, we prove that the Hopf map restricted to the affine Mahler Veronese $F(\Zp)$ is an isometry. Let $a,b \in F(\Zp) \ssq S^d$, and write
        \begin{align*}
            a &:= (1, a_1, \ldots, a_d) \\
            b &:= (1, b_1, \ldots, b_d).
        \end{align*}
    We have that the standard metric distance $d(a,b)$ is $\max_{1 \leq i \leq n} \abs{a_i-b_i}$. On the other hand, the distance between the projective points $\varphi(a), \varphi(b)$ is given by the maximum absolute value of the $2 \times 2$ minors of
        $
            \begin{bmatrix}
                1 & a_1 & \ldots & a_d \\
                1 & b_1 & \ldots & b_d
            \end{bmatrix}.
        $
    The maximum absolute value of the minors is invariant under $\GL_{d+1}(\Zp)$, so substracting integral multiples of the first column from the others yields the matrix
        $
            \begin{bmatrix}
                1 & 0 & \ldots & 0 \\
                1 & b_1-a_1 & \ldots & b_d-a_d
            \end{bmatrix}.
        $
    From which it is clear that $d(\varphi(a),\varphi(b)) = \norm{ a \wedge b} = d(a,b)$, showing that the Hopf map restricted to $F(\Zp)$ is an isometry.
\end{proof}

We now turn our attention to the complement of $\Zp$ in $\Qp$. We will make extensive use of the simple observation that if $\abs{t} > 1$, then $\abs{t-j} = \abs{t}$ for all $j \in \mathbb{Z}$.

\begin{lemma} \label{lem: valuations of binomial coefficients}
    For a fixed $\abs{t} > 1$, the function $g(k) := \abs{\binom{t}{k}}$ is strictly increasing for $k \in \mathbb{N}$.
\end{lemma}

\begin{proof}
    For $k>0$, we have that 
    \be
    \abs{\binom{t}{k}} = \abs{\frac{1}{k}} \cdot \abs{\frac{1}{(k-1)!}} \cdot \abs{\prod_{j=0}^{k-1}(t-j)} = \abs{\frac{t-k+1}{k}} \cdot \abs{\binom{t}{k-1}}.
    \ee
    Since $\abs{k} \leq 1$ and $\abs{t} > 1$, we are done.
\end{proof}

Define for $m \geq 1$ the $m$-th annulus $A_m := B(0; p^m) \bs B(0; p^{m-1})$ in $\Zp$. We show how to map $A_m$ to the unit sphere using a scaling of the Mahler Veronese map and we calculate the volume of the image.

\begin{definition}
    We extend the affine Mahler Veronese to all of $\Qp$ by
        \be
            \begin{tabu}{rccc}
            F \: & \Qp \bs \Zp &\rightarrow & S^d \\
                 & (t,1) &\mapsto &\left( \binom{t}{d}^{-1}, \binom{t}{1}\binom{t}{d}^{-1}, \ldots, 1 \right).
            \end{tabu}
        \ee
\end{definition}

By Lemma~\ref{lem: valuations of binomial coefficients}, we have that the image of $F$ is contained in $S^d$, so the map is well-defined. We also have that $F$ is injective, since it the $(d-1)$-th component is invertible on $\Qp \bs \Zp$. By our choice of scaling, the map $F$, together with the previously defined affine Mahler Veronese descends to a well-defined morphism $\nu_{\textrm{mahler}}\: \P^1 \rightarrow \P^d$.

\begin{lemma}
    If $0 \leq j < d$, then
        \be
            \frac{d}{dt} \binom{t}{j}\binom{t}{d}^{-1} = \frac{d!}{j!} \cdot \prod_{r=j}^{d-1} \frac{1}{t-r} \cdot \left( \sum_{r=j}^{d-1} \frac{1}{t-r} \right).
        \ee
\end{lemma}

\begin{proof}
    We observe that for $0 \geq j < d$ that
        \be
            \binom{t}{j}\binom{t}{d}^{-1} = \frac{d!}{j!} \cdot \frac{\prod_{r=j}^{d-1} (t-r)}{\prod_{r=0}^{d-1} (t-r)} = \frac{d!}{j!} \cdot \prod_{r=j}^{d-1} \frac{1}{t-r}.
        \ee
    The result follows from the product rule.
\end{proof}

\begin{corollary}
    If $t \in A_m$, then $\abs{J_F(t)} = \abs{d} \cdot p^{-2m}$.
\end{corollary}

\begin{proof}
    We see that the $d$-th component of the Jabobian is $0$, and for $j<d$ that
        \[
            \abs{(J_F(t))_j} = \abs{\frac{d!}{j!}} \cdot \abs{\prod_{r=j}^{d-1} \frac{1}{t-r} } \cdot \abs{ \sum_{r=j}^{d-1} \frac{1}{t-r} }.
        \]
    We see that each of the three factors is maximal when $j=d-1$, so the maximum is attained when $j=d-1$ and equals $\abs{d} \cdot \abs{t^{-1}} \cdot \abs{t^{-1}}$.
\end{proof}

\begin{corollary} \label{cor:mahler2}
    For $m \geq 1$, the image of $F\: A_m \rightarrow S^d$ has volume
        \[
            (p^m - p^{m-1}) \cdot \abs{d} \cdot p^{-2m} = \frac{\abs{d}}{p^m}(1-p^{-1})
        \]
    and the image of $F\: \Qp \bs \Zp \rightarrow S^d$ has volume $\frac{\abs{d}}{p}$. Finally, the Hopf fibration restricted to $F(\Qp \bs \Zp)$ is an isometry.
\end{corollary}

\begin{proof}
    Since $F$ is injective on $A_m$, and the Jacobian has constant absolute value, the first claim follows from Lemma~\ref{lemma:volcurve}. For the second, we have
        \be
            \vol_1 \left( F(\Qp \bs \Zp) \right) = \vol_1 \left( \bigcup_{m \geq 1} F(A_m) \right) = \abs{d}(1-p^{-1}) \cdot \sum_{m=1}^\infty \frac{1}{p^m} = \frac{\abs{d}}{p}.
        \ee
    Finally, as the last coordinate of $F(t)$ is $1$ for any $t \in \Qp \bs \Zp$, we see as before that the Hopf fibration restricts to an isometry on $F(\Qp \bs \Zp)$.
\end{proof}

With all of the calculation details in place, we can give several results at once about the expected number of zeros of random polynomials in the Mahler random model.

\begin{theorem}
    Let $f(t)$ be the random polynomial $f(t)=\sum_{k=0}^d \xi_k {\binom{t}{k}}$
    with $\{\xi_k\}_{k=0, \ldots, d}$ a family of i.i.d. random varuable unifomrly distributed on $\Zp$. Then:
        \begin{enumerate}
            \item
                
                    (Evans:) $\mathbb{E} \#\{f=0\}\cap \Zp=\frac{p^{\lfloor \log_p d\rfloor}}{1+p^{-1}}.$ \\
            \item
                $ \mathbb{E} \#\{f=0\} =\frac{p^{\lfloor \log_p d\rfloor} + \abs{d} p^{-1}}{1+p^{-1}}. $
                
        \end{enumerate}
\end{theorem}

\begin{proof}
    Reasoning as in the proof of Theorem \ref{thm:randompoly}, we see that, through the $p$-adic Integral Geometry Formula and the calculation of Corollary~\ref{cor:mahler} we have:
    \be \mathbb{E} \#\{f=0\}\cap \Zp=\frac{\mathrm{vol}_1(\varphi(F(\Zp))}{\mathrm{vol}_1(\P^1)}=\frac{p^{\lfloor \log_p d\rfloor}}{1+p^{-1}}.
    \ee
    For the second part, we use the calculation from Corollary~\ref{cor:mahler2} in a similar argument, and combine the result with part (1).
\end{proof}


\bibliographystyle{alpha}
\bibliography{padic}

\end{document}